\documentclass[11pt]{amsart}
\usepackage{amsmath}
\usepackage{graphicx, color, enumerate}
\usepackage{amscd}
\usepackage{amsmath}
\usepackage{amsfonts}
\usepackage{amssymb}
\usepackage{mathrsfs}
\usepackage{latexsym}
\usepackage[colorlinks=true]{hyperref}
\usepackage{tikz}
\usepackage{pgfplots}
\usepackage{amsfonts}
\pgfplotsset{compat=1.15}
\usepackage{mathrsfs}
\usetikzlibrary{arrows}

\newcommand\e\varepsilon

\newcommand\R{\mathbb R}

\newcommand\de\partial
\newcommand\weakto\rightharpoonup

\renewcommand\le\leqslant
\renewcommand\ge\geqslant
\renewcommand\a\alpha

\renewcommand\b\beta

\renewcommand\d\delta
\newcommand\vfi\varphi
\newcommand\g\gamma
\newcommand\gb\gamma
\renewcommand\l\lambda
\newcommand\n\nabla
\newcommand\s\sigma
\renewcommand\t\theta
\renewcommand\O\S

\newcommand\G\Gamma

\renewcommand\S\Sigma
\renewcommand\L\Lambda

\renewcommand\o\S

\def\bbm[#1]{\text{\boldmath $#1$}}
\newcommand\beq{\begin{equation}}
\newcommand\eeq{\end{equation}}

\renewcommand\leq{\leqslant}

\newtheorem{theorem}{Theorem}[section]
\newtheorem{lemma}[theorem]{Lemma}

\newtheorem{definition}[theorem]{Definition}
\newtheorem{proposition}[theorem]{Proposition}
\newtheorem{remark}[theorem]{Remark}

\def\sideremark#1{\ifvmode\leavevmode\fi\vadjust{\vbox to0pt{\vss
\hbox to0pt{\hskip\hsize\hskip1em
\vbox{\hsize3cm\tiny\raggedright\pretolerance10000
\noindent #1\hfill}\hss}\vbox to8pt{\vfil}\vss}}}

\definecolor{edu}{rgb}{0,0,1}

\numberwithin{equation}{section}

\title[Bound and ground states of  double critical ``NLS-KdV" systems]
{Bound and ground states of coupled ``NLS-KdV" equations with Hardy potential and critical power}


\author{Eduardo Colorado, Rafael L\'opez-Soriano, Alejandro Ortega}

\email[Eduardo Colorado ]{ecolorad@math.uc3m.es}%
\email[Rafael López-Soriano]{ralopezs@math.uc3m.es}%
\email[Alejandro Ortega ]{alortega@math.uc3m.es}
\address[E. Colorado, R. López-Soriano, A. Ortega]{Departamento de Matem\'aticas,
Universidad Carlos III de Madrid, Av. Universidad 30, 28911 Legan\'es (Madrid), Spain}

\begin{document}
\maketitle


\begin{center}
{\it   Dedicated to Antonio Ambrosetti in memoriam}
\end{center}

\begin{abstract}
We consider the existence of bound and ground states for a family of nonlinear elliptic systems in $\R^N$, which involves equations with critical power nonlinearities and Hardy-type singular potentials. The equations are coupled by what we call ``Schr\"odinger-Korteweg-de Vries'' non-symmetric terms, which arise in some phenomena of fluid mechanics. By means of variational methods, ground states are derived for several ranges of the positive coupling parameter $\nu$. Moreover, by using min-max arguments, we seek bound states under some energy assumptions.
\end{abstract}

\

\section{Introduction}
\setcounter{equation}0

In this work we study a system of elliptic equations involving critical power nonlinearities and Hardy-type singular potentials, coupled by the so-called ``Schr\"odinger-Korteweg-de Vries'' non-symmetric terms. Precisely,
we consider the problem

\beq\label{system:SKdV}
\left\{\begin{array}{ll}
-\Delta u - \lambda_1 \frac{u}{|x|^2}-u^{2^*-1}= 2 \nu h(x) uv  &\text{in }\mathbb{R}^N,\vspace{.3cm}\\
-\Delta v - \lambda_2 \frac{v}{|x|^2}-v^{2^*-1}= \nu h(x) u^2 &\text{in }\mathbb{R}^N,\vspace{.3cm}\\
u,v> 0 & \text{in }\mathbb{R}^N\setminus\{0\},
\end{array}\right.
\eeq
where $h \in L^{\infty}(\mathbb{R}^N)$ a positive function, $\lambda_1,\lambda_2\in(0,\Lambda_N)$ with $\Lambda_N=\frac{(N-2)^2}{4}$ the Hardy critical constant, $2^*=\frac{2N}{N-2}$ the critical Sobolev exponent and the coupling parameter $\nu>0$. In addition, we will assume that $3\le N\le 6$.

\

In the last  years, both coupled Nonlinear Schr\"odinger (NLS for short) equations and coupled NLS-Korteweg-de Vries (NLS-KdV) equations, have been extensively studied, (cf., e.g., \cite{AC2, ACR, BW, LW, MMP,POMP, SIR} and \cite{Col, DFO} repectively, among others). Systems  of coupled NLS equations arise naturally in Optics and also in the Hartree-Fock theory for Bose-Einstein
condensates, among other physical phenomena.
The main studied systems of Schr\"odinger equations adopt the form of the vector Schr\"odinger equation, $\displaystyle i\bf{E}_t+\bf{E}_{xx}+{\nu}|\bf{E}|^2\bf{E}=\bf{0}$
where $i$, $\bf{E}$ denote the imaginary unit and the complex envelope of an electrical field respectively, and ${\nu}>0$ (the coupling parameter) is a normalization constant corresponding to the fact that the medium is self-focusing. Written in coordinates, these systems are of the form
\beq\label{original:SKdV}
\left\{\begin{array}{ll}
if_t+f_{xx}+|f|^2f+{2\nu} fg= 0  &\text{in }\mathbb{R}{\times(0,\infty)},\vspace{.3cm}\\
g_t+g_{xxx}+gg_x+{\nu}(|f|^2)_x= 0 &\text{in }\mathbb{R}{\times(0,\infty)},
\end{array}\right.
\eeq
where $f=f(x,t)\in \mathbb{C}$, $g=g(x,t)\in \mathbb{R}$, and $\nu\in \mathbb{R}$ denotes the real coupling coefficient. System \eqref{original:SKdV} modelize the interaction of short and long dispersive waves for instance the interaction of capillary-gravity water waves (cf. \cite{AA,CL,FO} and the references therein).
Looking for solitary ``traveling-wave" solutions
$f(x,t)=e^{iwt}e^{ikx}u_1(x-ct)$, $g(x,t)=u_2(x-ct)$, with $u_j\ge 0$  real functions, and choosing  $\lambda_1=k^2+w$, $\lambda_2=2k$, we get the system
\begin{equation}\label{edo}
\left\{\begin{array}{rcl}
-u_1'' +\lambda_1 u_1 & =  u_1^3+{2\nu} u_1u_2 &\text{in }\mathbb{R},\vspace{.3cm} \\
-u_2'' +\lambda_2 u_2 &  \mkern-15mu =  \frac 12 u_2^2+ {\nu} u_1^2 &\text{in }\mathbb{R},
\end{array}\right.
\end{equation}
where the nonlinear coupling terms are knonw as \textit{non-symmetric Schr\"odinger-Korteweg-de Vries--type coupling}. In what concerns Hamiltonian systems with singular potentials we refer to \cite{BVG,FPR}.\newline
On the other hand, systems like \eqref{system:SKdV} have been studied in \cite{AbFePe,ChenZou} with similar coupling terms:
\beq\label{system:SKdV2}
\left\{\begin{array}{ll}
-\Delta u - \lambda_1 \dfrac{u}{|x|^2}-u^{2^*-1}= \nu \alpha h(x) u^{\alpha-1}v^\beta   &\text{in }\mathbb{R}^N,\vspace{.3cm}\\
-\Delta v - \lambda_2 \dfrac{v}{|x|^2}-v^{2^*-1}= \nu \beta h(x) u^\alpha v^{\beta-1} &\text{in }\mathbb{R}^N,\vspace{.3cm}\\
u,v> 0 & \text{in }\mathbb{R}^N\setminus\{0\},
\end{array}\right.
\eeq
where $\alpha,\beta> 1$. {The authors have recently established new existence results for bound and ground states of \eqref{system:SKdV2}. These results complement the given ones along this paper. See \cite{CLO} for a complete picture of the solvability of system \eqref{system:SKdV2}.}

\

Along this work, we will focus on the existence of positive solutions to system \eqref{system:SKdV} which has a
``Schr\"odinger-Korteweg-de Vries'' nonlinear non-symmetric terms similar to the one comming from the NLS-KdV system
\eqref{edo}. To do so, we shall use variational methods. In particular, let us recall that solutions to  \eqref{system:SKdV} are critical points of the following energy functional
\begin{equation}\label{funct:SKdV}
\begin{split}
\mathcal{J}_\nu (u,v)=&\frac{1}{2} \int_{\mathbb{R}^N} \left( |\nabla u|^2 + |\nabla v|^2  \right) \, dx -\frac{\lambda_1}{2} \int_{\mathbb{R}^N} \dfrac{u^2}{|x|^2} \, dx  -\frac{\lambda_2}{2} \int_{\mathbb{R}^N} \dfrac{v^2}{|x|^2}dx\\
&- \frac{1}{2^*} \int_{\mathbb{R}^N} \left( |u|^{2^*} + |v|^{2^*}  \right) \, dx -\nu \int_{\mathbb{R}^N} h(x) u^2 v \, dx,
\end{split}
\end{equation}
defined in $\mathbb{D}=\mathcal{D}^{1,2} (\mathbb{R}^N)\times
\mathcal{D}^{1,2}(\mathbb{R}^N)$, where $\mathcal{D}^{1,2}(\mathbb{R}^N)$ is the completion of $C^{\infty}_0(\mathbb{R}^N)$ under the norm
\begin{equation*}
\|u\|_{\mathcal{D}^{1,2}(\mathbb{R}^N)}^2=\int_{\mathbb{R}^N} \, |\nabla u|^2  \, dx.
\end{equation*}

A main role in our analysis will be performed by the unique \textit{semi-trivial} solution. Let us stress that for any $\nu \in \mathbb{R}$, problem \eqref{system:SKdV} has the \textit{semi-trivial} positive solution $(0,z_2)$, with $z_2$ satisfying the next problem
$$
-\Delta z_2 - \lambda_2 \frac{z_2}{|x|^2}=z_2^{2^*-1} \qquad \mbox{ and } \qquad z_2>0 \qquad \mbox{ in } \mathbb{R}^N\setminus\{ 0\}.
$$
Also some properties of the \textit{semi-trivial} pair $(z_1,0)$, with
$z_1$ satisfying
$$
-\Delta z_1 - \lambda_1 \frac{z_1}{|x|^2}=z_1^{2^*-1} \qquad \mbox{ and } \qquad z_1>0 \qquad \mbox{ in } \mathbb{R}^N\setminus\{ 0\},
$$
will be crucial in the analysis, although it is not a \textit{semi-trivial} solution.
By the study of the second variation of the energy functional $\mathcal{J}_\nu$, in Proposition~\ref{thm:semitrivial} is
proved the existence of an explicit parameter $\overline{\nu}>0$ which allows the couple $(0,z_2)$ to become either a local
 minimum if $\nu<\overline{\nu}$ or a saddle point in case that $\nu>\overline{\nu}$, as critical point of $\mathcal{J}_\nu$ on the Nehari manifold to be defined.

The parameter $\nu$ dramatically affects the behavior of $\mathcal{J}_\nu$: if $\nu>\overline{\nu}$, the \textit{semi-trivial} solution is a saddle point and it arises a positive ground state, see Theorem~\ref{thm:nugrande}; while in case that $\nu<\overline{\nu}$, the couple $(0,z_2)$ is a local minimum and the energy configuration depends on $\lambda_1,\lambda_2$.
 
The relation between $\lambda_1$ and $\lambda_2$ controls the relation between the energy levels of the  \textit{semi-trivial} solution and $(z_1,0)$: if $\lambda_1\ge \lambda_2$, we find a positive ground state, see Theorem~\ref{thm:lambdaground}; if $\lambda_2>\lambda_1$ and $\nu$ is small enough, then the ground state corresponds to $(0,z_2)$, see Theorem~\ref{thm:groundstates}; while, under the assumption that $\lambda_1$ and $\lambda_2$ are somehow closed, we prove that the energy functional has a {\it Mountain-Pass} geometry on the Nehari manifold, so that a positive bound state is found, see Theorem~\ref{thm:MPgeom}.

To prove the above mentioned results, we first need to establish some compactness properties.  This step is accomplished by Palais-Smale (PS for short) condition relying on the classical \textit{concentration-compactness principle} by Lions (cf. \cite{Lions1,Lions2}). To that end, we have to take into account the failure of the compactness of the embedding of $\mathcal{D}^{1,2}(\mathbb{R}^N)$ in $L^{2^*}(\mathbb{R}^N)$.
Moreover, the coupling term $u^2v$ might be critical depending on the dimension $N$. We shall distinguish between the subcritical dimensions, $3\le N\le 5$, and the critical one, $N=6$. Then, we will assume along the paper that $3\le N\leq 6$.

\
The paper has three more sections. Section \ref{section2} contains the main functional setting and definitions, as well as an analysis of the character as a critical point of the \textit{semi-trivial} solution. In
Section \ref{section:PS}, we prove the PS condition in both  subcritical and critical dimensions.
Finally, Section \ref{section4} is devoted to prove the main results about the existence of bound and ground states of
\eqref{system:SKdV}.

\section{Variational setting}\label{section2}
The energy functional associated to \eqref{system:SKdV} is given by $\mathcal{J}_\nu$ introduced in \eqref{funct:SKdV}. $\mathcal{J}_\nu$ is well defined in $\mathbb{D}=\mathcal{D}^{1,2} (\mathbb{R}^N)\times \mathcal{D}^{1,2} (\mathbb{R}^N)$, endowed with the norm $\displaystyle \|(u,v)\|^2_{\mathbb{D}}=\|u\|^2_{\lambda_1}+\|v\|^2_{\lambda_2}$,
\begin{equation*}
\|u\|^2_{\lambda}=\int_{\mathbb{R}^N} |\nabla u|^2 \, dx - \lambda \int_{\mathbb{R}^N} \frac{u^2}{|x|^2} \, dx.
\end{equation*}

\

\noindent Note that, by Hardy's inequality,
\begin{equation}\label{hardy_inequality}
\Lambda_N \int_{\R^N} \dfrac{u^2}{|x|^2} \, dx \leq \int_{\R^N} |\nabla u|^2 \, dx,
\end{equation}
the norm $\|\cdot\|_{\lambda}$ is equivalent to the norm $\|\cdot\|_{\mathcal{D}^{1,2} (\mathbb{R}^N)}$ for any $\lambda\in (0,\Lambda_N)$.

\

\noindent On the other hand, if either system \eqref{system:SKdV} is decoupled, namely $\nu=0$, or the first component vanishes, then the second component $v$ is a solution of the entire equation
\beq\label{entire}
-\Delta z - \lambda \frac{z}{|x|^2}=z^{2^*-1} \qquad \mbox{ with }   \qquad z>0  \qquad \mbox{ in } \mathbb{R}^N\setminus\{0\}.
\eeq
Observe that if the second component $v=0$, then necessarily  $u=0$ because of the second equation of \eqref{system:SKdV}. That is the reason why there exists only one \textit{semi-trivial} solution.
Positive solutions to equation \eqref{entire} were completely classified by Terracini, (cf.  \cite{Terracini}).
In particular, among other results, it was proved that, if $\lambda\in\left(0,\Lambda_N\right)$, the family of solutions to equation \eqref{entire} is given by
\beq\label{zeta}
z_\mu^{\lambda}(x)= \mu^{-\frac{N-2}{2}}z_1^{\lambda}\left(\frac{x}{\mu}\right) \qquad \mbox{ with } \qquad z_1^{\lambda}(x)=\dfrac{A(N,\lambda)}{|x|^{a_\lambda}\left(1+|x|^{2-\frac{4a_\lambda}{N-2}}\right)^{\frac{N-2}{2}}},
\eeq
with $a_\lambda=\frac{N-2}{2}-\sqrt{\left( \frac{N-2}{2}\right)^2-\lambda}$ and $A(N,\lambda)=\frac{N(N-2-2a_\lambda)^2}{N-2} $.
Solutions of \eqref{entire} are also minimizers of the associated Rayleigh quotient
\beq\label{Slambda}
\mathcal{S}(\lambda)= \inf_{\substack{u\in \mathcal{D}^{1,2}(\mathbb{R}^N)\\
u\not\equiv0}}\frac{\|u\|^2_{\lambda}}{\|u_\mu^\lambda\|_{2^*}^{2}}=
\frac{\|z_\mu^\lambda\|^2_{\lambda}}{\|z_\mu^\lambda\|_{2^*}^{2}}= \left(1-\frac{4\lambda}{(N-2)^2} \right)^{\frac{N-1}{N}}
\mathcal{S}=\left(1-\frac{\lambda}{\Lambda_N} \right)^{\frac{N-1}{N}}
\mathcal{S},
\eeq
with $\mathcal{S}$ being the Sobolev's constant, i.e.,
\begin{equation}\label{sobolev_inequality}
\mathcal{S}\int_{\R^N} |u|^{2^*}dx \leq \int_{\R^N} |\nabla u|^2dx.
\end{equation}
Using \eqref{zeta}, it is easy to see that
\beq\label{normcrit}
\|z_\mu^\lambda\|_{2^*}^{2^*}= \mathcal{S}^{\frac{N}{2}}(\lambda),
\eeq
and, as a consequence, for every $\mu>0$ the pair $(0,z_\mu^{\lambda_2})$ is a \textit{semi-trivial} solution of \eqref{system:SKdV}. Our main aim is then to find neither \textit{semi-trivial} nor trivial solutions, namely solutions $(u,v)$ with $u\not\equiv 0$ and $v\not\equiv 0$ in $\mathbb{R}^N$.
\begin{definition}\label{bound-ground} A pair $(u,v)\in\mathbb{D}$ is said to be a non-trivial {\it bound state} of \eqref{system:SKdV} if it is a non-trivial critical point of $\mathcal{J}_\nu$.
While a bound state $(\tilde{u},\tilde{v})$ is called a ground state if its energy is minimal among all the non-trivial and non-negative bound states, i.e.,
\begin{equation}\label{ctilde}
\tilde{c}_\nu=\mathcal{J}_\nu(\tilde{u},\tilde{v})=\min\{\mathcal{J}_\nu(u,v): (u,v)\in \mathbb{D}\setminus \{ (0,0)\},\; (u,v)\ge (0,0), \mbox{ and } \mathcal{J}_\nu'(u,v)=0\}.
\end{equation}
\end{definition}

The functional $\mathcal{J}_\nu \in \mathcal{C}^2(\mathbb{D},\mathbb{R})$ and $\mathcal{J}_\nu$ is unbounded from below, namely, given $(\tilde u,\tilde v)\in \mathbb{D}$, if $\int_{\R^N} h(x) \tilde u^2 \tilde v \, dx > 0$, then $\displaystyle\mathcal{J}_\nu(t \tilde u,t \tilde v) \to -\infty$ as $t\to\infty$. Therefore,
it is convenient to introduce a proper constraint in order to minimize the energy functional $\mathcal{J}_\nu$. To that end, let us define the Nehari manifold associated to $\mathcal{J}_\nu$ as
\begin{equation*}
\mathcal{N}_\nu=\left\{ (u,v) \in \mathbb{D}\setminus \{ (0,0)\} \, : \,  \Psi(u,v)=0 \right\},
\end{equation*}
where $\displaystyle \Psi(u,v)=\left\langle \mathcal{J}'_\nu(u,v){\big|}(u,v)\right\rangle$. Given $(u,v) \in \mathcal{N}_\nu$,  it holds
\begin{equation} \label{Nnueq1}
 \|(u,v)\|_\mathbb{D}^2=\int_{\mathbb{R}^N} \left( |u|^{2^*} + |v|^{2^*} \right)dx+3\nu \int_{\mathbb{R}^N} h(x) u^2 vdx,
\end{equation}
and
\beq\label{Nnueq2}
\mathcal{J}_{\nu}{\big|}_{\mathcal{N}_\nu} (u,v) = \frac{1}{N} \int_{\mathbb{R}^N} \left(  |u|^{2^*} +  |v|^{2^*}  \right) \, dx + \frac{\nu}{2}  \int_{\mathbb{R}^N} h(x) u^{2} v  \, dx.
\eeq
For every $(u,v)\in\mathbb{D}\setminus \{ (0,0)\}$, there exists a constant $t$ depending on $(u,v)$ such that $(tu,tv) \in \mathcal{N}_\nu$. Indeed, $t_{(u,v)}$ is the unique real solution to the algebraic equation
\beq\label{normH}
\|(u,v)\|_\mathbb{D}^2=t^{2^*-2}  \int_{\mathbb{R}^N} \left(  |u|^{2^*} + |v|^{2^*}  \right) \, dx  + 3\nu \, t \int_{\mathbb{R}^N} h(x) u^{2} v  \, dx.
\eeq
By using \eqref{Nnueq1}, one gets that
\beq\label{criticalpoint1}
\begin{split}
\mathcal{J}_\nu''(u,v)[u,v]^2&=\left\langle \Psi'(u,v){\big|}(u,v)\right\rangle\\
&=-\|(u,v)\|_{\mathbb{D}}^2-(2^*-1)\int_{\mathbb{R}^N} \left(|u|^{2^*} + |v|^{2^*} \right)dx <0,
\end{split}
\eeq
for any $(u,v) \in \mathcal{N}_\nu$. Therefore, $\mathcal{N}_\nu$ is a locally smooth manifold close to every $(u,v)\in \mathbb{D}\setminus \{ (0,0)\}$ with $\Psi(u,v)=0$. In addition,
\begin{equation*}
 \mathcal{J}_\nu''(0,0)[\varphi_1,\varphi_2]^2=\|(\varphi_1,\varphi_2)\|^2_{\mathbb{D}}>0 \quad \text{ for any } (\varphi_1,\varphi_2)\in \mathcal{N}_\nu.
\end{equation*}
Then, $(0,0)$ is a strict minimum for $\mathcal{J}_\nu$ and, thus, it is an isolated point of the set $\displaystyle \mathcal{N}_\nu  \, \cup \, \{ (0,0)\}$. Consequently, the Nehari manifold $\mathcal{N}_\nu$ is a smooth complete manifold of codimension $1$. Furthermore, there exists $\rho>0$ constant such that
\beq\label{criticalpoint2}
\|(u,v)\|_{\mathbb{D}} > \rho\quad\text{for all } (u,v)\in \mathcal{N}_\nu.
\eeq

Let us emphasize that, if $(u,v) \in \mathcal{N}_\nu$ is a critical point of $\mathcal{J}_\nu$ constrained on $\mathcal{N}_\nu$, there exists a Lagrange multiplier $\omega$ such that
\begin{equation*}
\nabla_{\mathcal{N}_\nu}\mathcal{J}_{\nu}(u,v)=\mathcal{J}'_\nu(u,v)-\omega \Psi'(u,v)=0.
\end{equation*}
Testing this expression with $(u,v)$, one gets $\Psi (u,v)=\left\langle \mathcal{J}'_\nu(u,v){\big|}(u,v)\right\rangle = \omega \left\langle \Psi'_\nu(u,v){\big|}(u,v)\right\rangle=0$. By using \eqref{criticalpoint1}, we deduce $\left\langle \Psi'(u,v){\big|}(u,v)\right\rangle<0$. So, $\omega=0$ and hence $\mathcal{J}'_\nu(u,v)=0$. In conclusion,
\vspace{0.15cm}
\begin{equation}\label{eq:critt}
(u,v) \in \mathbb{D}\mbox{ is a critical point of }\mathcal{J}_\nu\Longleftrightarrow (u,v) \in \mathcal{N}_\nu\mbox{ is a critical point of }\mathcal{J}_{\nu}\mbox{ on }\mathcal{N}_\nu.
\end{equation}
Let us also note that, the functional $\mathcal{J}_{\nu}$ on the Nehari manifold $\mathcal{N}_\nu$ reads also as
\beq\label{Nnueq}
\mathcal{J}_{\nu}{\big|}_{\mathcal{N}_\nu} (u,v) = \frac{1}{6} \|(u,v)\|^2_{\mathbb{D}} + \frac{6-N}{6N} \int_{\mathbb{R}^N} \left(  |u|^{2^*} +  |v|^{2^*}  \right) \, dx.
\eeq
Hence, by \eqref{criticalpoint2} and $N\le 6$, we have $ \displaystyle
\mathcal{J}_{\nu} (u,v) > \frac{1}{6} \rho^2$ for all $\displaystyle (u,v)\in \mathcal{N}_\nu$.
Thus, $\mathcal{J}_{\nu}$ is bounded from below on $\mathcal{N}_\nu$, so we can look for solutions of \eqref{system:SKdV} by minimizing the functional on $\mathcal{N}_\nu$.

\subsection{Semi-trivial solution}

\

In this subsection we are going to study the character of the \textit{semi-trivial solution} as critical point of $\mathcal{J}_\nu|_{\mathcal{N}_\nu}$. Let us consider the decoupled energy functionals $\mathcal{J}_i:\mathcal{D}^{1,2} (\mathbb{R}^N)\rightarrow\mathbb{R}$,
\begin{equation}\label{funct:Ji}
\mathcal{J}_i(u) =\frac{1}{2} \int_{\mathbb{R}^N}  |\nabla u|^2 \, dx -\frac{\lambda_i}{2} \int_{\mathbb{R}^N} \dfrac{u^2}{|x|^2}  \, dx - \frac{1}{2^*} \int_{\mathbb{R}^N} |u|^{2^*} \, dx,
\end{equation}
for $i=1,2$ so that $\displaystyle \mathcal{J}_\nu(u,v)=\mathcal{J}_1(u)+\mathcal{J}_2(v)-\nu \int_{\mathbb{R}^N} h(x) u^2 v \, dx$.
Observe that $z_\mu^{\lambda_i}$, defined by \eqref{zeta}, is a global minimum of $\mathcal{J}_i$ constrained on the  Nehari manifold $\mathcal{N}_i$ defined by
\begin{equation}\label{Nnui}
\begin{split}
\mathcal{N}_i&= \left\{ u \in \mathcal{D}^{1,2} (\mathbb{R}^N) \setminus \{0\} \, : \,  \left\langle \mathcal{J}'_i(u){\big|} u\right\rangle=0 \right\}\\
&= \left\{ u \in\mathcal{D}^{1,2} (\mathbb{R}^N) \setminus \{0\} \, : \,  \|u\|_{\lambda_i}=\int_{\mathbb{R}^N} |u|^{2^*} \, dx \right\}.
\end{split}
\end{equation}
Due to the explicit expression \eqref{zeta}, it is easy to prove that the energy levels of $z_\mu^{\lambda_i}$,
are
\beq\label{Jzeta}
\mathcal{J}_1(z_\mu^{\lambda_1})=\dfrac{1}{N}\mathcal{S}^{\frac{N}{2}}(\lambda_1)=\mathcal{J}_\nu(z_\mu^{\lambda_1},0), \qquad \mathcal{J}_2(z_\mu^{\lambda_2})=\dfrac{1}{N}\mathcal{S}^{\frac{N}{2}}(\lambda_2)=\mathcal{J}_\nu(0,z_\mu^{\lambda_2}),
\eeq
for any $\mu>0$ with $\mathcal{S}(\lambda)$ defined in \eqref{Slambda}.\newline
Given $(\tilde{u},\tilde{v})\in\mathcal{N}_\nu$ we denote by $T_{(\tilde{u},\tilde{v})} \, \mathcal{N}_\nu$  the tangent space of $\mathcal{N}_\nu$ at $(\tilde{u},\tilde{v})$. Note that
\beq\label{TNnu2}
\varphi=(\varphi_1,\varphi_2) \in T_{(0,z_\mu^{\lambda_2})} \, \mathcal{N}_\nu\Longleftrightarrow\varphi_2 \in T_{z_\mu^{\lambda_2}} \, \mathcal{N}_2.
\eeq

Next, we determine the character of $(0,z_\mu^{\lambda_2})$ as critical point of $\mathcal{J}_\nu|_{\mathcal{N}_\nu}$.
\begin{proposition}\label{thm:semitrivial}
There exits $\overline{\nu}>0$ such that the following holds:
\begin{enumerate}
\item[i)] if $0<\nu<\overline{\nu}$,  $(0,z_\mu^{\lambda_2})$ is a local minimum of $\mathcal{J}_\nu$ constrained on $\mathcal{N}_\nu$,
\item[ii)] for any $\nu>\overline{\nu}$,  $(0,z_\mu^{\lambda_2})$ is a saddle point of $\mathcal{J}_\nu$
constrained on  $\mathcal{N}_\nu$.
\end{enumerate}

\end{proposition}

\begin{proof}
To obtain $i)$, let us set
\beq\label{overnu}
\overline{\nu}=\inf_{\substack{\varphi\in \mathcal{D}^{1,2}(\mathbb{R}^N)\\
\varphi\not\equiv0}}\frac{\|\varphi\|^2_{\lambda_1}}{\displaystyle2\int_{\R^N} h(x) \varphi^2 z_\mu^{\lambda_2} \, dx}.
\eeq
Next, given $\varphi=(\varphi_1,\varphi_2)\in T_{(0,z_\mu^{\lambda_2})} \, \mathcal{N}_\nu$, we have
\beq\label{secondvar}
\mathcal{J}_\nu''(0,z_\mu^{\lambda_2})[(\varphi_1,\varphi_2)]^2= \|\varphi_1\|^2_{\lambda_1}+\mathcal{J}_{2}''(z_\mu^{\lambda_2})[\varphi_2]^2-2\nu\int_{\R^N} h(x) \varphi_1^2 z_\mu^{\lambda_2}  \, dx.
\eeq
As $z_\mu^{\lambda_2}$ is a minimum of $\mathcal{J}_{2}$ on $\mathcal{N}_2$ and $\varphi_2\in T_{z_\mu^{\lambda_2}} \, \mathcal{N}_2$, by \eqref{TNnu2}, there exists $C>0$ such that
\beq\label{secondmin}
\mathcal{J}_{2}''(z_\mu^{\lambda_2})[\varphi_2]^2\ge C \|\varphi_2\|^2_{\lambda_2}.
\eeq
Then, if $\nu<\overline{\nu}$, there exists $c>0$ such that $\displaystyle \mathcal{J}_\nu''(0,z_\mu^{\lambda_2})[(\varphi_1,\varphi_2)]^2\ge c (\| \varphi_1\|^2_{\lambda_1} + \|\varphi_2\|^2_{\lambda_2} )$, which proves that $(0,z_\mu^{\lambda_2})$ is a local strict minimum of $\mathcal{J}_\nu$ constrained on $\mathcal{N}_\nu$.

To prove $ii)$, first we note that, by \eqref{secondvar} and \eqref{secondmin},
\begin{equation}\label{saddle0}
\mathcal{J}_\nu''(0,z_\mu^{\lambda_2})[(0,\varphi_2)]^2=\mathcal{J}_{2}''(z_\mu^{\lambda_2})[\varphi_2]^2\ge C \| \varphi_2\|^2_{\lambda_2}.
 \end{equation}
On the other hand, if we take $\varphi=(\varphi_1,0)$ such that
\begin{equation*}
\nu> \frac{\|\varphi_1\|^2_{\lambda_1}}{\displaystyle 2\int_{\R^N} h(x) \varphi_1^2 z_\mu^{\lambda_2}  \, dx }>\overline{\nu},
\end{equation*}
we get
\beq\label{saddle2}
\mathcal{J}_\nu''(0,z_\mu^{\lambda_2})[(\varphi_1,0)]^2=\|\varphi_1\|^2_{\lambda_1}-2\nu\int_{\R^N} h(x) \varphi_1^2 z_\mu^{\lambda_2}  \, dx<0 \qquad \mbox{ for any } \nu>\overline{\nu}.
\eeq
Thus, by \eqref{saddle0} and \eqref{saddle2}, we conclude that $(0,z_\mu^{\lambda_2})$ is saddle point of $\mathcal{J}_\nu $ on $\mathcal{N}_\nu$.
\end{proof}

\begin{remark}\label{zeta1Nehari}
Although the pair $(z_\mu^{\lambda_1},0)$ is not a critical point of the energy functional $\mathcal{J}_\nu$, this couple does belong to the Nehari manifold $\mathcal{N}_\nu$.
\end{remark}

To conclude this section we recall the following result which will be useful in several proofs.
\begin{lemma}\label{algelemma}{\cite[Lemma 3.3]{AbFePe}}
Assume that $A, B>0$ and $\gamma \ge 2$. We define the set
\begin{equation*}
\Sigma_\nu=\{\sigma \in (0,+\infty)  \, : \, A \sigma^{\frac{N-2}{N}} < \sigma + B \nu \sigma^{\frac{\gamma}{2} \frac{N-2}{N}} \}.
\end{equation*}
Then, for any $\varepsilon>0$ there exists $\tilde{\nu}>0$ such that, for $0<\nu<\tilde{\nu}$, we have $\displaystyle \inf_{\Sigma_\nu} \sigma > (1-\varepsilon) A^{\frac{N}{2}}$.
\end{lemma}

\section{The Palais-Smale condition}\label{section:PS}

As commented in the introduction, a crucial step to obtain existence of solution  to \eqref{system:SKdV} is the PS condition.

\begin{definition}
Let $V$ be a Banach space. We say that $\{u_n\} \subset V$ is a PS sequence for an energy functional $\mathfrak{F}:V\rightarrow\mathbb{R}$ if
\begin{equation}\label{PSc}
\mathfrak{F}(u_n)\to c\quad\mbox{and}\quad  \mathfrak{F}'(u_n) \to 0\quad\mbox{in}\ V^*\quad \hbox{as}\quad n\to + \infty,
\end{equation}
where $V^*$ is the dual space of $V$. Moreover, we say that $\{u_n\}$ satisfies a PS condition if
\begin{equation*}
\{u_n\}\quad \mbox{has a strongly convergent subsequence.}
\end{equation*}
\end{definition}
Even more, we say that $\{u_n\}\subset V$ is a PS sequence at level $c$ if  \eqref{PSc} holds. Also, the functional
$\mathfrak{F}$ satisfies the PS  condition at level $c$ if every PS sequence at level $c$
 for $\mathfrak{F}$ satisfies the PS condition.

\begin{lemma}\label{critical Constrained}
  Assume that $\{(u_n,v_n)\}\subset \mathcal{N}_\nu$ is a PS sequence of $\mathcal{J}_{\nu}$ constrained on $\mathcal{N}_\nu$. Then  $\{(u_n,v_n)\}$ is a PS sequence of $\mathcal{J}_\nu$.
\end{lemma}
\begin{proof}
Since $\{(u_n,v_n)\}\subset \mathcal{N}_\nu$ is a PS sequence of $\mathcal{J}_\nu$ constrained on $\mathcal{N}_\nu$ we have
$$\mathcal{J}_\nu (u_n,v_n)\to c\qquad\text{and}\qquad\nabla_{\mathcal{N}_\nu}\mathcal{J}_\nu (u_n,v_n) =\mathcal{J}'_\nu(u_n,v_n)-\omega_n\Psi' (u_n,v_n)\to 0,$$
where $\omega_n$ is the corresponding Lagrange multiplier sequence. Testing the above expression with $(u_n,v_n)$, we have $\Psi (u_n,v_n)=(\mathcal{J}'_nu (u_n,v_n)|(u_n,v_n))=0$, while by \eqref{criticalpoint1}
$\Psi' (u_n,v_n)<0$, then we conclude that $\omega_n\to 0$. As a consequence, we obtain
$\mathcal{J}'_\nu (u_n,v_n)\to 0$.
\end{proof}
\begin{remark}
By Lemma~\ref{critical Constrained} and \eqref{eq:critt},
 it is enough to show that the PS condition for $\mathcal{J}_\nu$
 holds instead of proving the PS condition for $\mathcal{J}_\nu|_{\mathcal{N}_\nu}$.
\end{remark}
Now, we address the boundedness of PS sequences that, together with the compact embedding of the space $\mathcal{D}^{1,2}$ in the subcritical regime, will provide compactness of PS sequences.

\begin{lemma}\label{lemmaPS0}
If $\{(u_n,v_n)\} \subset \mathbb{{D}}$ is a PS sequence for $\mathcal{J}_\nu$ at level
$c\in\mathbb{R}$, then $\|(u_n,v_n)\|_{\mathbb{D}}<C$.
\end{lemma}

\begin{proof}
Let $\{(u_n,v_n)\} \subset \mathbb{D}$ be a PS sequence for $\mathcal{J}_{\nu}$ at level $c$, i.e.,
\begin{equation*}
\mathcal{J}_{\nu}(u_n,v_n)\to c\quad\text{and}\quad \mathcal{J}_{\nu}'(u_n,v_n)\to0\quad\text{as }n\to+\infty.
\end{equation*}
Since $\mathcal{J}_{\nu}'(u_n,v_n)\to0$ in $\mathbb{D}'$, we have $\displaystyle \left\langle \mathcal{J}_{\nu}'(u_n,v_n)\left|\frac{(u_n,v_n)}{\|(u_n,v_n)\|_{\mathbb{D}}} \right.\right\rangle\to0$.
Hence, there exists a subsequence (still denoted by $\{(u_n,v_n)\}$) such that
\begin{equation*}
\|(u_n,v_n)\|_{\mathbb{D}}^2-\int_{\mathbb{R}^N}\left(|u_n|^{2^*}+|v_n|^{2^*}\right)dx-3\nu\int_{\mathbb{R}^N}
h(x)u_n^2v_ndx=\|(u_n,v_n)\|_{\mathbb{D}}\cdot o(1).
\end{equation*}
Since $\mathcal{J}_{\nu}(u_n,v_n)\to c$, one obtains
\begin{equation*}
\frac12\|(u_n,v_n)\|_{\mathbb{D}}^2-\frac{1}{2^*}\int_{\mathbb{R}^N}\left(|u_n|^{2^*}+|v_n|^{2^*}\right)dx-\nu
\int_{\mathbb{R}^N}h(x)u_n^2v_ndx=c+o(1).
\end{equation*}
Therefore
\begin{equation}\label{eq:limit}
\mathcal{J}_{\nu}(u_n,v_n)-\frac13\left\langle \mathcal{J}_{\nu}'(u_n,v_n)\left|\frac{(u_n,v_n)}{\|(u_n,v_n)\|_{\mathbb{D}}}
\right.\right\rangle=c+\|(u_n,v_n)\|_{\mathbb{D}}\cdot o(1),
\end{equation}
and, hence,
\begin{equation}\label{eq:limit2}
\frac{1}{6}\|(u_n,v_n)\|_{\mathbb{D}}^2+\frac{6-N}{6N}\int_{\mathbb{R}^N}\left(|u_n|^{2^*}+|v_n|^{2^*}\right)dx=c+\|(u_n,v_n)\|_{\mathbb{D}}\cdot o(1).
\end{equation}
As a consequence, $\displaystyle \frac16\|(u_n,v_n)\|_{\mathbb{D}}^2\leq c+\|(u_n,v_n)\|_{\mathbb{D}}\cdot o(1)$. Thus, the sequence $\{(u_n,v_n)\}$ is bounded in $\mathbb{D}$.
\end{proof}
\subsection{Subcritical dimension $3\le N\le5$}
\begin{lemma}\label{lemmaPS2}
Assume $3\le N\le5$.  Then,  $\mathcal{J}_\nu$ satisfies the PS condition at every level $c$ satisfying
\beq\label{hyplemmaPS2}
c<\frac{1}{N} \min\{ \mathcal{S}^{\frac{N}{2}}(\lambda_1),\mathcal{S}^{\frac{N}{2}}(\lambda_2) \}.
\eeq
\end{lemma}
\begin{proof}

Because of Lemma~\ref{lemmaPS0}, any PS sequence is bounded in $\mathbb{D}$ so that there exists $(\tilde{u},\tilde{v})\in\mathbb{D}$ and a subsequence (denoted also by $\{(u_n,v_n)\}$) such that
 \begin{align*}\
(u_n,v_n) \rightharpoonup (\tilde u,\tilde v)& \quad \hbox{weakly in  } \mathbb{D},\\
(u_n,v_n) \to (\tilde u,\tilde v)&\quad \hbox{strongly in  } L^q(\mathbb{R}^N)\times L^q(\mathbb{R}^N)\text{ for } 1\leq q<2^*,\\
(u_n,v_n) \to (\tilde u,\tilde v)&\quad \hbox{a.e. in  }\mathbb{R}^N.
\end{align*}
By the \textit{concentration-compactness principle} (cf. \cite{Lions1,Lions2}), there exist a subsequence
(denoted also by) $\{(u_n,v_n)\}$, two (at most countable) sets of points $\{x_j\}_{j\in\mathfrak{J}}\subset\mathbb{R}^N$ and $\{y_k\}_{k\in\mathfrak{K}}\subset\mathbb{R}^N$, and non-negative quantities $\{\mu_j,\rho_j\}_{j\in\mathfrak{J}}$, $\{\overline{\mu}_k,\overline{\rho}_k\}_{k\in\mathfrak{K}}$, $\mu_0$, $\rho_0$, $\gamma_0$, $\overline{\mu}_0$, $\overline{\rho}_0$ and $\overline{\gamma}_0$ such that

\begin{equation}\label{con-comp}
\left\{
\begin{array}{rl}
|\nabla u_n|^2\mkern-10mu &\rightharpoonup d\mu\ge |\nabla \tilde u|^2+\sum_{j\in\mathfrak{J}}\mu_j\delta_{x_j}+\mu_0\delta_0,\\
\\
|\nabla v_n|^2\mkern-10mu  &\rightharpoonup d\overline{\mu}\ge |\nabla \tilde v|^2+\sum_{k\in\mathfrak{K}}\overline{\mu}_k\delta_{y_k}+\overline{\mu}_0\delta_0,\\
\\
|u_n|^{2^*} \mkern-15mu &\rightharpoonup d\rho= |\tilde u|^{2^*}+\sum_{j\in\mathfrak{J}}\rho_j\delta_{x_j}+\rho_0\delta_0,\\
\\
|v_n|^{2^*} \mkern-15mu  &\rightharpoonup d\overline{\rho}= |\tilde v|^{2^*}+\sum_{k\in\mathfrak{K}}\overline{\rho}_k\delta_{y_k}+\overline{\rho}_0\delta_0,\\
\\
\dfrac{u_n^2}{|x|^2} \mkern-10mu &\rightharpoonup d\gamma=\dfrac{\tilde u^2}{|x|^2}+\gamma_0\delta_0,\\
\\
\dfrac{v_n^2}{|x|^2} \mkern-10mu &\rightharpoonup d\overline{\gamma}=\dfrac{\tilde v^2}{|x|^2}+\overline{\gamma}_0\delta_0,
\end{array}
\right.
\end{equation}
in the sense of measures. Let us note that, using \eqref{sobolev_inequality} and \eqref{hardy_inequality}, the
above numbers satisfy
\begin{equation}\label{ineq:sobcon}
\mathcal{S}\rho_j^{\frac{2}{2^*}}\leq\mu_j\quad\text{for all }j\in\mathfrak{J}\cup\{0\}\qquad\text{and}\qquad
\mathcal{S}\overline{\rho}_k^{\frac{2}{2^*}}\leq\overline{\mu}_k\quad\text{for all }k\in\mathfrak{K}\cup\{0\},
\end{equation}
\begin{equation}\label{ineq:harcon}
\Lambda_N\gamma_0\leq\mu_0\qquad\text{and}\qquad \Lambda_N\overline{\gamma}_0\leq\overline{\mu}_0.
\end{equation}
The concentration of $\{u_n\}$ at infinity is described by the quantities
\begin{equation}\label{con-compinfty}
\begin{split}
\mu_{\infty}&=\lim\limits_{R\to+\infty}\limsup\limits_{n\to+\infty}\int_{|x|>R}|\nabla u_n|^{2}dx,\\
\rho_{\infty}&=\lim\limits_{R\to+\infty}\limsup\limits_{n\to+\infty}\int_{|x|>R}|u_n|^{2^*}dx,\\
\gamma_{\infty}&=\lim\limits_{R\to+\infty}\limsup\limits_{n\to+\infty}\int_{|x|>R}\frac{u_n^{2}}{|x|^2}dx.
\end{split}
\end{equation}
The concentration at infinity of $\{v_n\}$ is given by $\overline{\mu}_{\infty}$, $\overline{\rho}_{\infty}$ and $\overline{\gamma}_{\infty}$ defined analogously.
For $j\in\mathfrak{J}$, we consider $\varphi_{j,\varepsilon}(x)$ a smooth cut-off function centered at $x_j$, i.e., $\varphi_{j,\varepsilon}\in C^{\infty}(\mathbb{R})$ and
\begin{equation}\label{cutoff}
\varphi_{j,\varepsilon}=1 \quad \hbox{in}\quad B_{\frac{\varepsilon}{2}}(x_j),\quad \varphi_{j,\varepsilon}=0 \quad \hbox{in}
\quad B_{\varepsilon}^c(x_j)\quad \hbox{and}\quad\displaystyle|\nabla \varphi_{j,\varepsilon}|\leq \frac{4}{\varepsilon},
\end{equation}
where $B_r(x_j)$ denotes the ball of radius $r>0$ centered at $x_j\in\mathbb{R}^N$. Therefore, testing $\mathcal{J}_{\nu}'(u_n,v_n)$ with $(u_n\varphi_{j,\varepsilon},0)$, we get
\begin{equation*}
\begin{split}
0&=\lim\limits_{n\to+\infty}\left\langle \mathcal{J}_{\nu}'(u_n,v_n)\big|(u_n\varphi_{j,\varepsilon},0)\right\rangle\\
&=\lim\limits_{n\to+\infty}\left(\int_{\mathbb{R}^N}|\nabla u_n|^2\varphi_{j,\varepsilon}dx+\int_{\mathbb{R}^N}u_n\nabla u_n\nabla\varphi_{j,\varepsilon}dx-\lambda_1\int_{\mathbb{R}^N}\frac{u_n^2}{|x|^2}\varphi_{j,\varepsilon}dx\right.\\
&\mkern+80mu-\left.\int_{\mathbb{R}^N}|u_n|^{2^*}\varphi_{j,\varepsilon}dx-2\nu\int_{\mathbb{R}^N}h(x)u_n^2v_n \varphi_{j,\varepsilon} dx\right)\\
&=\int_{\mathbb{R}^N}\varphi_{j,\varepsilon}d\mu+\int_{\mathbb{R}^N}\tilde u\nabla \tilde u\nabla\varphi_{j,\varepsilon}dx-\lambda_1\int_{\mathbb{R}^N}\varphi_{j,\varepsilon}d\gamma\\
&\mkern+25mu-\int_{\mathbb{R}^N}\varphi_{j,\varepsilon}d\rho-2\nu\int_{\mathbb{R}^N}h(x)\tilde u^2\tilde v \varphi_{j,\varepsilon} \, dx.
\end{split}
\end{equation*}
Observe that $0\notin supp(\varphi_{j,\varepsilon})$ for every $\varepsilon>0$. Since $h\in L^{\infty}(\mathbb{R}^N)$, taking $\varepsilon\to0$, it follows that $\mu_j-\rho_j\leq0$. Therefore, it arises the following alternative:
\begin{equation}\label{afr}
\text{Either }\rho_j=0 \text{ for all  } j\in\mathfrak{J}\quad \text{ or, by \eqref{ineq:sobcon},}\quad\rho_j\ge \mathcal{S}^{\frac{N}{2}}\quad\text{for all } j\in\mathfrak{J},
\end{equation}
that is, either the PS sequence has a convergent subsequence or it concentrates around some of the points $x_j$ and, therefore, the set $\mathfrak{J}$ is finite.\newline
An analogous argument provides the same conclusion for the numbers $\overline{\rho}_k$, i.e.,
\begin{equation}\label{afr2}
\text{Either }\overline{\rho}_k=0 \text{ for all  } k\in\mathfrak{K}\quad \text{ or, by \eqref{ineq:sobcon},}\quad\overline{\rho}_j\ge \mathcal{S}^{\frac{N}{2}}\quad\text{for all } k\in\mathfrak{K},
\end{equation}
and the set  $\mathfrak{K}$ is also finite.\newline
Testing $\mathcal{J}_{\nu}'(u_n,v_n)$ with $(u_n\varphi_{0,\varepsilon},0)$ where $\varphi_{0,\varepsilon}$ denotes a smooth cut-off function centered at $x=0$, it follows that $\mu_0-\lambda_1\gamma_0-\rho_0\leq 0$ and $\overline{\mu}_0-\lambda_2\overline{\gamma}_0-\overline{\rho}_0\leq0$.
From \eqref{Slambda} we get
\begin{equation}\label{ineq:con0_a}
\mu_0-\lambda_1\gamma_0\ge \mathcal{S}(\lambda_1)\rho_0^{\frac{2}{2^*}}\qquad\text{and}\qquad \overline{\mu}_0-\lambda_2\overline{\gamma}_0\ge \mathcal{S}(\lambda_2)\overline{\rho}_0^{\frac{2}{2^*}},
\end{equation}
so that, by \eqref{ineq:harcon},
\begin{equation}\label{ineq:con0}
\rho_0=0\quad\text{or}\quad \rho_0\ge \mathcal{S}^{\frac{N}{2}}(\lambda_1)\qquad\text{and}\qquad
\overline{\rho}_0=0\quad\text{or}\quad \overline{\rho}_0\ge \mathcal{S}^{\frac{N}{2}}(\lambda_2).
\end{equation}
Next, for $R>0$ such that $\{x_j\}_{j\in\mathfrak{J}}\cup \{0\} \subset  B_{R}(0)$, we consider $\varphi_{\infty,\varepsilon}$ a cut-off function supported near $\infty$, i.e.,
\begin{equation}\label{cutoffinfi}
\varphi_{\infty,\varepsilon}=0 \quad \hbox{in}\quad B_{R}(0),\quad \varphi_{\infty,\varepsilon}=1 \quad \hbox{in}
\quad B_{R+1}^c(0)\quad \hbox{and}\quad\displaystyle|\nabla \varphi_{\infty,\varepsilon}|\leq \frac{4}{\varepsilon}.
\end{equation}

Testing $\mathcal{J}_{\nu}'(u_n,v_n)$ with $(u_n\varphi_{\infty,\varepsilon},0)$ being $\varphi_{\infty,\varepsilon}$ a smooth cut-off function supported in a neighborhood of $\infty$ we can analogously prove that $\mu_{\infty}-\lambda_1\gamma_{\infty}-\rho_{\infty}\leq 0$ as well as $\overline{\mu}_{\infty}-\lambda_2\overline{\gamma}_{\infty}-\overline{\rho}_{\infty}\leq0$ and, as above, we get
\begin{equation}\label{ineq:coninf_a}
\mu_{\infty}-\lambda_1\gamma_{\infty}\ge \mathcal{S}(\lambda_1)\rho_{\infty}^{\frac{2}{2^*}}\qquad\text{and}\qquad\overline{\mu}_{\infty}-\lambda_2\overline{\gamma}_{\infty}\ge \mathcal{S}(\lambda_2)\overline{\rho}_{\infty}^{\frac{2}{2^*}},
\end{equation}
and we also conclude
\begin{equation}\label{ineq:coninf}
\rho_{\infty}=0\quad\text{or}\quad \rho_{\infty}\ge \mathcal{S}^{\frac{N}{2}}(\lambda_1)\qquad\text{and}\qquad\overline{\rho}_{\infty}=0\quad\text{or}\quad \overline{\rho}_{\infty}\ge \mathcal{S}^{\frac{N}{2}}(\lambda_2).
\end{equation}
From \eqref{eq:limit2} we get
\begin{equation*}
c=\frac16\|(u_n,v_n)\|_{\mathbb{D}}^2+\frac{6-N}{6N}\int_{\mathbb{R}}\left(|u_n|^{2^*}+|v_n|^{2^*}\right)dx+o(1)\quad\text{as}\ n\to+\infty.
\end{equation*}
Hence, by \eqref{con-comp}, \eqref{ineq:sobcon}, \eqref{ineq:harcon}, \eqref{ineq:con0_a} and \eqref{ineq:coninf_a} above, we get

\begin{equation}\label{ineq:larga}
\begin{split}
c\ge& \frac16\left(\|(\tilde  u,\tilde v)\|_{\mathbb{D}}^2+ \sum_{j\in\mathfrak{J}} \mu_j+ (\mu_0-\lambda_1\gamma_0)+(\mu_{\infty}-\lambda_1\gamma_{\infty})\right. \\
&\mkern+120mu +\left. \sum_{k\in\mathfrak{K}} \overline{\mu}_k+ (\overline{\mu}_0-\lambda_2\overline{\gamma}_0)+(\overline{\mu}_{\infty}-\lambda_2\overline{\gamma}_{\infty})\right)\\
&+\frac{6-N}{6N}\left(\int_{\mathbb{R}^N}|\tilde u|^{2^*}dx+\int_{\mathbb{R}^N}|\tilde v|^{2^*}dx \right. \\
& \mkern+120mu + \left.  \sum_{j\in\mathfrak{J}}\rho_j+\rho_0+\rho_{\infty}+\sum_{k\in\mathfrak{K}}\overline{\rho}_k+\overline{\rho}_0+\overline{\rho}_{\infty} \right) \\
\ge&\frac16\left( \mathcal{S}\left[\sum_{j\in\mathfrak{J}}\rho_j^{\frac{2}{2^*}}+\sum_{k\in\mathfrak{K}}\overline{\rho}_k^{\frac{2}{2^*}}\right]+\mathcal{S}(\lambda_1)\left[\rho_0^{\frac{2}{2^*}}+\rho_{\infty}^{\frac{2}{2^*}}\right]+\mathcal{S}(\lambda_2)\left[\overline{\rho}_0^{\frac{2}{2^*}}+\overline{\rho}_{\infty}^{\frac{2}{2^*}}\right]\right)\\
&+\frac{6-N}{6N}\left(\sum_{j\in\mathfrak{J}}\rho_j+\rho_0+\rho_{\infty}+\sum_{k\in\mathfrak{K}}\overline{\rho}_k+\overline{\rho}_0+\overline{\rho}_{\infty} \right).
\end{split}
\end{equation}
If concentration at the point $x_j$, i.e., $\rho_j>0$ occurs, from above and \eqref{afr}, it follows that
\begin{equation*}
c\ge\frac16\mathcal{S}^{1+\frac{N}{2}\frac{2}{2^*}}+\frac{6-N}{6N}\mathcal{S}^{\frac{N}{2}}=\frac{1}{N}\mathcal{S}^{\frac{N}{2}},
\end{equation*}
which contradicts the hypothesis \eqref{hyplemmaPS2} on the energy level $c$. Therefore, $\rho_j=\mu_j=0$ for every $j\in\mathfrak{J}$. In a  similar way, we also conclude that $\overline{\rho}_k=\overline{\mu}_k=0$ for every $k\in\mathfrak{K}$.\newline
If $\rho_0\neq0$, from the above inequalities and \eqref{ineq:con0}, we infer that
\begin{equation*}
c\ge\frac{1}{N}S^{\frac{N}{2}}(\lambda_1),
\end{equation*}
which also contradicts the hypothesis \eqref{hyplemmaPS2} on the energy level $c$. Hence, $\rho_0=0$. Analogously we also find that $\overline{\rho}_0=0$. Finally, arguing as above and using \eqref{ineq:coninf} we also find $\rho_{\infty}=0$ and $\overline{\rho}_{\infty}=0$. Thus, the PS sequence has a subsequence that strongly converges in $L^{2^*}(\mathbb{R}^N)\times L^{2^*}(\mathbb{R}^N)$, i.e., it satisfies the PS-condition. Finally, note that by the strongly convergence in $L^{2^*}(\mathbb{R}^N)\times L^{2^*}(\mathbb{R}^N)$,
\begin{equation*}
\|(u_n-\tilde u,v_n-\tilde v)\|_{\mathbb{D}}^2 = \left\langle \mathcal{J}_\nu'(u_n,v_n)\left| (u_n-\tilde u,v_n-\tilde v) \right.\right\rangle + o(1),
\end{equation*}
and the strongly $\mathbb{D}$-convergence of $\{(u_n,v_n)\}$ follows.
\end{proof}

The next Lemma~\ref{lemmaPS1} is a refinement of Lemma~\ref{lemmaPS2}, in the sense that it states the PS condition for supercritical energy levels excluding multipliers or combinations of the critical ones.

In order to address the issue of positive solutions, it will be useful to consider the problem
\beq\label{systemp}
\left\{\begin{array}{ll}
-\Delta u - \lambda_1 \dfrac{u}{|x|^2}-(u^+)^{2^*-1}= 2 \nu h(x) u^+\, v  &\text{in }\mathbb{R}^N,\vspace{.3cm}\\
-\Delta v - \lambda_2 \dfrac{v}{|x|^2}-(v^+)^{2^*-1}= \nu h(x)  (u^+)^{2} &\text{in }\mathbb{R}^N,
\end{array}\right.
\eeq
where  $u^+=\max\{u,0\}$. Similarly, $u^-=\min\{u,0\}$ denotes the negative part of the function $u$. With this notation,
 $u=u^+ + u^-$.

It is not difficult to prove that the pair $(u,v)$ solution to \eqref{systemp} is positive in every component. Moreover, the system \eqref{systemp} is a variational system and its solutions are critical points of the energy functional
\begin{equation}\label{funct:SKdVp}
\mathcal{J}^+_\nu (u,v)=\|(u,v)\|^2_{\mathbb{D}}
- \frac{1}{2^*} \int_{\mathbb{R}^N} \left( (u^+)^{2^*} + (v^+)^{2^*} \, dx  \right)-\nu \int_{\mathbb{R}^N} h(x) (u^+)^2 \, v, \, dx
\end{equation}
defined in $\mathbb{D}$. We will denote $\mathcal{N}^+_\nu$ as the Nehari manifold associated to $\mathcal{J}^+_\nu $, i.e.,
\begin{equation*}
\mathcal{N}^+_\nu=\left\{ (u,v) \in \mathbb{D}\setminus \{ (0,0)\} \, : \,  \left\langle (\mathcal{J}_\nu^+)' (u,v){\big|}(u,v)\right\rangle=0 \right\}.
\end{equation*}
\begin{lemma}\label{lemmaPS1}
Assume that $3\le N\le5$, $\lambda_2 \ge \lambda_1$ and
\begin{equation}\label{PS0}
\mathcal{S}^{\frac{N}{2}}(\lambda_1)+\mathcal{S}^{\frac{N}{2}}(\lambda_2)<\mathcal{S}^{\frac{N}{2}}.
\end{equation}
There exists $\tilde{\nu}>0$ such that for $0<\nu\leq\tilde{\nu}$ and $\{(u_n,v_n)\} \subset \mathbb{D}$  a PS sequence for $\mathcal{J}^+_\nu$ at level $c\in\mathbb{R}$ such that
\beq\label{PS1}
\frac{1}{N} \mathcal{S}^{\frac{N}{2}}(\lambda_2)<c<\frac{1}{N} \left(\mathcal{S}^{\frac{N}{2}}(\lambda_1)+\mathcal{S}^{\frac{N}{2}}(\lambda_2) \right),
\eeq
and
\beq\label{PS2}
c\neq \frac{\ell}{N} \mathcal{S}^{\frac{N}{2}}(\lambda_2) \quad \mbox{ for every } \ell \in \mathbb{N}\setminus \{0\},
\eeq
then $(u_n,v_n)\to(\tilde u,\tilde v) \in \mathbb{D}$ up to subsequence.
\end{lemma}

\begin{proof}

As in Lemma~\ref{lemmaPS0}, any PS sequence for $\mathcal{J}_\nu^+$ is also
bounded in $\mathbb{D}$ and, hence, there exists a subsequence $\{(u_n,v_n)\}$ which weakly
converges to $(\tilde{u},\tilde{v}) \in \mathbb{D}$. Since $(\mathcal{J}^+_\nu)'(u_n,v_n)\to 0$, then
\begin{equation*}
\left\langle (\mathcal{J}^+_\nu)'(u_n,v_n)\left| (u_n^-,0) \right.\right\rangle= \int_{\mathbb{R}^N} |\nabla u_n^-|^2 \, dx - \lambda_1 \int_{\mathbb{R}^N} \dfrac{(u_n^-)^2}{|x|^2} \, dx  \to0,
\end{equation*}
and, hence, that $u_n^-\to 0$ strongly in $\mathcal{D}^{1,2} (\mathbb{R}^N)$. Analogously,
 \begin{equation*}
\left\langle (\mathcal{J}^+_\nu)'(u_n,v_n)\left| (0,v_n^-) \right.\right\rangle= \int_{\mathbb{R}^N} |\nabla v_n^-|^2 \, dx - \lambda_2 \int_{\mathbb{R}^N} \dfrac{(v_n^-)^2}{|x|^2} \, dx   -\nu \int_{\mathbb{R}^N} h(x) (u^+)^2 \, v^- \, dx \ \to0,
\end{equation*}
so that $v_n^-\to 0$. As a consequence, $\{(u^+_n,v^+_n)\}$ is a bounded PS sequence of $\mathcal{J}^+_\nu$. Thus, we can assume that $\{(u_n,v_n)\}$ is a non-negative PS sequence for$\mathcal{J}_\nu$ at the level $c$.

Next, a similar argument to that of Lemma~\ref{lemmaPS2}
provides the existence of a subsequence, still denoted by $\{(u_n,v_n)\}$, two (at most countable)
sets of points $\{x_j\}_{j\in\mathfrak{J}}\subset\mathbb{R}^N$ and $\{y_k\}_{k\in\mathfrak{K}}\subset\mathbb{R}^N$, and also
non-negative quantities $\{\mu_j,\rho_j\}_{j\in\mathfrak{J}}$, $\{\overline{\mu}_k,\overline{\rho}_k\}_{k\in\mathfrak{K}}$, $\mu_0$,
$\rho_0$, $\gamma_0$, $\overline{\mu}_0$, $\overline{\rho}_0$ and $\overline{\gamma}_0$ such that \eqref{con-comp} is satisfied. Besides, the inequalities \eqref{afr}, \eqref{afr2}, \eqref{ineq:con0_a},
\eqref{ineq:con0} hold.

Similarly, we define the concentration at infinity with the values $\mu_\infty$, $\rho_\infty$, $\overline{\mu}_\infty$ and $\overline{\rho}_\infty$ as in \eqref{con-compinfty},
for which \eqref{ineq:coninf_a} and \eqref{ineq:coninf} hold.

\

\textbf{Claim}:
\beq\label{claimPS21}
\mbox{ Either } u_n\to  \tilde u \mbox{ in } L^{2^*}(\mathbb{R}^N) \qquad \mbox{ or } \qquad v_n\to \tilde v \mbox{ in } L^{2^*}(\mathbb{R}^N).
\eeq

Let us prove the claim arguing by contradiction. Assume that $u_n$ and $v_n$ do not converge strongly in $L^{2^*}(\mathbb{R}^N)$. Then, there exists $j\in \mathfrak{J}\cup \{0 \cup \infty\}$ and $k\in \mathfrak{J}\cup \{0,\infty\}$ such that $\rho_{j}>0$ and $\overline{\rho}_k>0$. Finally, because of \eqref{eq:limit2}, \eqref{afr}, \eqref{afr2}, \eqref{ineq:con0_a},  \eqref{ineq:con0} and \eqref{ineq:larga} we get
\begin{equation*}
\begin{split}
c=&\frac16\|(u_n,v_n)\|_{\mathbb{D}}^2+\frac{6-N}{6N}\int_{\mathbb{R}^N}\left(u_n^{2^*}+v_n^{2^*}\right)dx+o(1),\\
\ge&\frac16\left(\mathcal{S}(\lambda_1)\rho_J^{\frac{2}{2^*}} +\mathcal{S}(\lambda_2) \overline{\rho}_K^{\frac{2}{2^*}}\right)+\frac{6-N}{6N}\left(\rho_J+\overline{\rho}_K \right)\\
\ge& \ \frac{1}{N} \left(\mathcal{S}^{\frac{N}{2}}(\lambda_1) + \mathcal{S}^{\frac{N}{2}}(\lambda_2)  \right),
\end{split}
\end{equation*}
which contradicts assumption \eqref{PS1}, so claim \eqref{claimPS21} is proved.

\

Subsequently, we claim that:
\beq\label{claimPS22}
\mbox{ either } u_n\to  \tilde u \mbox{ in } \mathcal{D}^{1,2}(\mathbb{R}^N) \qquad \mbox{ or } \qquad v_n\to \tilde  v \mbox{ in } \mathcal{D}^{1,2}(\mathbb{R}^N).
\eeq

Without loss of generality, we assume by \eqref{claimPS21} that $u_n$ strongly converges in $L^{2^*}(\mathbb{R}^N)$. Then, it is enough to observe that
\begin{equation*}
\|(u_n- \tilde u)\|_{\lambda_1}^2 = \left\langle \mathcal{J}_\nu'(u_n,v_n)\left| (u_n- \tilde u,0) \right.\right\rangle + o(1).
\end{equation*}
This implies that $u_n\to u$ in $\mathcal{D}^{1,2}(\mathbb{R}^N)$. Repeating the argument for $v_n$, completes \eqref{claimPS22}.\newline
In order to show that both components strongly converge in $\mathcal{D}^{1,2}(\mathbb{R}^N)$ we consider two cases:

\

\textbf{Case 1}: $v_n$ strongly converges to $\tilde{v}$ in $\mathcal{D}^{1,2}(\mathbb{R}^N)$.

In order to prove that $u_n$ strongly converges to $\tilde u$ in $\mathcal{D}^{1,2}(\mathbb{R}^N)$, let us assume, by contradiction, that none of its subsequences converge. Note that, assuming $\mathfrak{J}\cup \{0,\infty\}$ contains more
than one point, because of \eqref{ineq:larga},  \eqref{afr}, \eqref{ineq:con0_a}, \eqref{ineq:con0}, \eqref{ineq:coninf_a} and \eqref{ineq:coninf} it follows that
\begin{equation*}
c\ge \frac{2}{N} \mathcal{S}^{\frac{N}{2}}(\lambda_1)\ge \frac{1}{N} \left( \mathcal{S}^{\frac{N}{2}}(\lambda_1) +
 \mathcal{S}^{\frac{N}{2}}(\lambda_2), \right)
\end{equation*}
since $\lambda_2\ge \lambda_1$ and $\mathcal{S}(\lambda)$ is decreasing. This expression contradicts \eqref{PS1}. Then, assume that there exists only one concentration point for the sequence $u_n$, corresponding to the index $j\in \mathfrak{J}\cup \{0,\infty\}$.

Let us prove now that $\tilde{v}\not \equiv 0$. Assume that $\tilde v\equiv 0, $ then $\tilde u\equiv 0$ and hence $u_n$ satisfies
\begin{equation*}
-\Delta u_n - \lambda_1 \frac{u_n}{|x|^2}-{u_n}^{2^*-1}=o(1)
\end{equation*}
in the dual space $\displaystyle\left( \mathcal{D}^{1,2} (\mathbb{R}^N)\right)^*$ and
\begin{equation*}
c=\mathcal{J}_\nu(u_n,v_n)+o(1)=\frac{1}{N} \int_{\R^N} u_n^{2^*}+o(1)\to \frac{1}{N} \rho_j,
\end{equation*}
since $u_n$ concentrates at one point $x_j$. Moreover, since $j\in \mathfrak{J}$, then $u_n$ is a positive PS sequence for the functional
$$
\mathcal{I}_j(u)=\frac{1}{2} \int_{\R^N} |\nabla u|^2\, dx -\frac{1}{2^*} \int_{\R^N} |u|^{2^*} \, dx.
$$
Hence, by the characterization of PS sequences for $\mathcal{I}_j$ provided by \cite{Stru}, we have
$\rho_j=\ell \mathcal{S}^{\frac{N}{2}}$ for some $\ell\in\mathbb{N}$, in contradiction with \eqref{PS0} and \eqref{PS1}.
So that $\mathfrak{J}=\emptyset$. If $u_n$ concentrates at zero or infinity we can use a similar
argument for $\mathcal{J}_1$, defined in \eqref{funct:Ji}, together with the results of \cite{Smets} to conclude
$$
c=\mathcal{J}_\nu(u_n,v_n)+o(1)=\mathcal{J}_1(u_n)+o(1)\to \frac{\ell}{N} \mathcal{S}^{\frac{N}{2}}(\lambda_1),
$$
with $\ell \in \mathbb{N}\cup \{0\}$. This is in contradiction with \eqref{PS1}. Then, ${v}\ge 0$ in $\mathbb{R}$. Next, we prove that $u_n\weakto\tilde{u}$ in $D^{1,2}(\R^N)$ with
$\tilde {u}\not \equiv 0$. As before, by contradiction, we assume that $\tilde{u}=0$ so that $\tilde {v}$ satisfies
\beq\label{vtildeq}
-\Delta \tilde {v} - \lambda_2 \frac{\tilde {v}}{|x|^2}={\tilde v}^{2^*-1} \qquad \mbox{ in } \mathbb{R}^N.
\eeq
Then ${v}=z_\mu^{\lambda_2}$ for some $\mu>0$ and $\displaystyle\int_{\R^N}\tilde  {v}^{2^*} \, dx =\mathcal{S}^{\frac{N}{2}}(\lambda_2)$ by
\eqref{normcrit}. Hence, combining \eqref{ineq:larga} with \eqref{afr}, \eqref{ineq:con0_a},
\eqref{ineq:con0}, it follows that
$$
c\ge \frac{1}{N}\left( \int_{\R^N}\tilde {v}^{2^*} \, dx +\mathcal{S}^{\frac{N}{2}}(\lambda_1)  \right) = \frac{1}{N}\left( \mathcal{S}^{\frac{N}{2}}(\lambda_1)+ \mathcal{S}^{\frac{N}{2}}(\lambda_2) \right),
$$
which contradicts \eqref{PS1}. Therefore, $\tilde u,\tilde v \not \equiv 0$. Next,
\begin{equation}\label{eqlemmaPS10}
\begin{split}
c&=\mathcal{J}_\nu(u_n,v_n)-\frac{1}{2}\left\langle \mathcal{J}_\nu'(u_n,v_n)\left| (u_n,v_n) \right.\right\rangle + o(1)\\
& = \frac{1}{N} \int_{\R^N} \left(  u_n^{2^*} + v_n^{2^*} \right) \, dx + \frac{\nu}{2} \int_{\R^N} h(x) u_n^2 v_n \, dx + o(1) \to \\
&\frac{1}{N} \int_{\R^N} \left( \tilde {u}^{2^*} +\tilde {v}^{2^*}  \right) \, dx +  \frac{\rho_j}{N}+   \frac{\nu}{2} \int_{\R^N} h(x) \tilde {u}^2 \tilde {v} \, dx .
\end{split}
\end{equation}
by the concentration at $j \in \mathfrak{J}\cup \{ 0, \infty\}$. Since $\left\langle \mathcal{J}_\nu'(u_n,v_n)\left| (\tilde {u},\tilde {v}) \right.\right\rangle \to 0$, we find
$$
\| (\tilde {u},\tilde {v}) \|_{\mathbb{D}}= \int_{\R^N} \left( \tilde {u}^{2^*} +\tilde v^{2^*} \right) \, dx + 3\nu \int_{\R^N} h(x)\tilde {u}^2 \tilde v \, dx,
$$
that is the same to say $(\tilde {u},\tilde {v}) \in\mathcal{N}_\nu$. Next, by \eqref{eqlemmaPS10}, \eqref{eqlemmaPS11}, \eqref{Nnueq}, \eqref{afr}, \eqref{ineq:con0_a}, \eqref{ineq:con0} and  \eqref{ineq:larga}, we have
\begin{equation*}
\begin{split}
&\mathcal{J}_\nu(\tilde{u},\tilde{v})=\frac{1}{N}\int_{\R^N} \left(  \tilde {u}^{2^*} + \tilde {v}^{2^*} \right) \, dx + {\frac{1}{2}}\nu \int_{\R^N} h(x) \tilde{u}^2 \tilde{v} \, dx \\
& = c - \frac{\rho_j}{N} < \frac{1}{N}\left( \mathcal{S}^{\frac{N}{2}}(\lambda_1)+ \mathcal{S}^{\frac{N}{2}}(\lambda_2) \right) - \frac{1}{N} \mathcal{S}^{\frac{N}{2}}(\lambda_1) =  \frac{1}{N} \mathcal{S}^{\frac{N}{2}}(\lambda_2).
\end{split}
\end{equation*}
Then,
\begin{equation*}
\tilde{c}_\nu= \inf_{(u,v)\in\mathcal{N}_\nu} \mathcal{J}_\nu(u,v) < \frac{1}{N} \mathcal{S}^{\frac{N}{2}}(\lambda_2),
\end{equation*}
that, for $\nu$ sufficiently small, contradicts Theorem~\ref{thm:groundstates}. Thus, $u_n\to \tilde{u}$ strongly in $\mathbb{D}^{1,2}(\R^N)$.

\

\textbf{Case 2}: $u_n$ strongly converges to $\tilde{u}$ in $\mathcal{D}^{1,2}(\mathbb{R}^N)$.

As before, in order to prove that $v_n$ strongly converges to $\tilde{v}$ in $\mathcal{D}^{1,2}(\mathbb{R}^N)$, let us assume, by contradiction, that none of its subsequences converge. First, let us prove that $\tilde{u}\not \equiv 0$. If we assume, once again by contradiction, that
$\tilde{u} \equiv 0$, then $v_n$ is a PS sequence for $\mathcal{J}_2$ defined in \eqref{funct:Ji} at level $c$. As $v_n\weakto\tilde{v}$ in $\mathcal{D}^{1,2}(\R^N)$ with $\tilde{v}$ solution to \eqref{vtildeq}, we have $\tilde{v}=z_\mu^{\lambda_2}$ for some $\mu>0$. Furthermore, because of the compactness theorem given by \cite{Smets}, it follows that
\begin{equation}\label{eq:cero}
c= \mathcal{J}_2 (v_n)  + o(1) \to   \mathcal{J}_2 (z_\mu^{\lambda_2})+ \frac{m}{N} \mathcal{S}^{\frac{N}{2}}+ \frac{\ell}{N} \mathcal{S}^{\frac{N}{2}}(\lambda_2)=\frac{m}{N} \mathcal{S}^{\frac{N}{2}}+ \frac{\ell+1}{N} \mathcal{S}^{\frac{N}{2}}(\lambda_2),
\end{equation}
with $m\in \mathbb{N}$ and $\ell \in \mathbb{N}\cup \{0\}$, in contradiction with \eqref{PS1} and \eqref{PS2}. Hence, $\tilde{u} \not \equiv 0$.

Conversely, assuming that $\tilde v \equiv 0$, we have $\tilde{u}\equiv 0$ by the second equation of \eqref{system:SKdV}, which gives a contradiction with \eqref{eq:cero}. Thus, $\tilde{u},\tilde{v}\not \equiv 0$. Since $(\tilde{u},\tilde{v})$ is a solution of \eqref{system:SKdV}, we get
\begin{equation}\label{eqlemmaPS11}
{ \mathcal{J}_\nu(\tilde{u},\tilde{v})=\frac{1}{N}\int_{\R^N} \left(  \tilde{u}^{2^*} + \tilde{v}^{2^*} \right) \, dx + \frac{\nu}{2} \int_{\R^N} h(x) \tilde{u}^2 \tilde{v} \, dx \leq c.}
\end{equation}

Since by assumption $v_n$ does not strongly converge in $\mathcal{D}^{1,2}(\R^N)$, using again \eqref{eqlemmaPS10}, it follows that there exists at least one $k\in\mathfrak{K}\cup \{0,\infty\}$ such that $\overline{\rho}_k>0$
$$
c= \frac{1}{N} \left( \int_{\R^N}  \left(  \tilde{u}^{2^*} + \tilde{v}^{2^*} \right)  \, dx + \sum_{k\in\mathfrak{K}} \overline{\rho}_k+ \overline{\rho}_0 + \overline{\rho}_\infty \right)  + \frac{\nu}{2} \int_{\R^N} h(x) \tilde{u}^2 \tilde{v} \, dx .
$$
By \eqref{eqlemmaPS11}, \eqref{afr2}, \eqref{ineq:con0_a}, \eqref{ineq:con0} and \eqref{PS1}, one gets
\beq\label{eqlemmaPS12}
\begin{split}
\mathcal{J}_\nu(\tilde{u},\tilde{v})&= c - \frac{1}{N} \sum_{k\in\mathfrak{K}} \overline{\rho}_k+ \overline{\rho}_0 + \overline{\rho}_\infty \\
& < \frac{1}{N}\left( \mathcal{S}^{\frac{N}{2}}(\lambda_1)+ \mathcal{S}^{\frac{N}{2}}(\lambda_2) \right) - \frac{1}{N} \mathcal{S}^{\frac{N}{2}}(\lambda_2)\\
& =  \frac{1}{N} \mathcal{S}^{\frac{N}{2}}(\lambda_1). \end{split}
\eeq
Using the first equation of \eqref{system:SKdV} and the definition of $\mathcal{S}^{\frac{N}{2}}(\lambda_1)$, we find
\beq\label{eqlemmaPS13}
\sigma_1 + \nu \int_{\R^N} h(x) \tilde{u}^2 \tilde{v} \, dx = \int_{\R^N} |\nabla \tilde{u}|^2 \, dx - \lambda_1 \int_{\R^N} \frac{\tilde{u}^2}{|x|^2} \, dx \ge \mathcal{S}^{\frac{N}{2}}(\lambda_1) \sigma^{2/2^*}_1,
\eeq
where $\sigma_1= \int_{\R^N} \tilde{u}^{2^*} \, dx$. Using Hölder's inequality, one gets
\beq\label{Holder}
\int_{\mathbb{R}^N} h(x)  \,  \tilde{u}^{2} \,  \tilde{v}  \, dx \leq ||h||_{L^{\infty}(\mathbb{R}^N)} \left( \int_{\mathbb{R}^N} \tilde{u}^{2^*}   \, dx \right)^{\frac{2}{2^*}}\left( \int_{\mathbb{R}^N} \tilde{v}^{2^*}  \, dx   \right)^{\frac{1}{2^*}}.
\eeq

Combining \eqref{Holder} and \eqref{eqlemmaPS11}, we can transform \eqref{eqlemmaPS13} into
\beq\label{eqlemmaPS14}
\sigma_1 + C \nu \sigma_1^\frac{2}{2^*}\ge \mathcal{S}(\lambda_1) \sigma_1^\frac{2}{2^*}.
\eeq

Since $\tilde{v}\not \equiv 0$, there exits $\tilde \varepsilon$ such that  $\int_{\R^N} \tilde{v}^{2^*} \, dx\ge \tilde \varepsilon$. Taking  $\varepsilon>0$ such that $\tilde \varepsilon\ge \varepsilon \mathcal{S}^{\frac{N}{2}}(\lambda_1)$, because of \eqref{eqlemmaPS14} and Lemma~\ref{algelemma}, we find some $\tilde \nu>0$ such that
$$
\sigma_1\ge (1-\varepsilon)\mathcal{S}^{\frac{N}{2}}(\lambda_1) \quad \mbox{ for any } 0<\nu\leq \tilde{\nu}.
$$
The above estimates and \eqref{eqlemmaPS11}, provide us with
$$
\mathcal{J}_\nu(\tilde{u},\tilde{v}) \ge \frac{1}{N} \left((1-\varepsilon) \mathcal{S}^{\frac{N}{2}}(\lambda_1)+ \tilde{\varepsilon} \right)\ge  \frac{1}{N} \mathcal{S}^{\frac{N}{2}}(\lambda_1),
$$
which contradicts \eqref{eqlemmaPS12}. Hence, $v_n\to \tilde{v}$ strongly in $\mathcal{D}^{1,2}(\R^N)$.
\end{proof}

\subsection{Critical dimension $N=6$}

\

In the critical case, more hypothess on the function $h$ are supposed:
\beq\label{hypH}\tag{H}
h\in L^{\infty}(\mathbb{R}^N), \, h \mbox{ continuous around $0$ and $\infty$ and } h(0)=\lim_{x\to+\infty} h(x)=0.
\eeq
We also split the results in the cases in which either $h$ is radial or $h$ is non-radial but $\nu>0$ is sufficiently small.

To obtain the existence of {\it Mountain-Pass} solutions claimed in Theorem~\ref{thm:MPgeom} for the critical regime, we need the following Lemma, analogous to \cite[Lemma 4.1]{AbFePe}.

\begin{lemma}\label{lemcritic}
Assume that $N=6$ and \eqref{hypH} holds. Let $\{(u_n,v_n)\} \subset \mathbb{D}_r$ be a PS sequence for $\mathcal{J}_\nu$ at level $c\in\mathbb{R}$ such that either \eqref{hyplemmaPS2} or \eqref{PS1} and \eqref{PS2} hold, then there exists $\overline{\nu}>0$ such that for every $\nu\leq \overline{\nu}$ then $(u_n,v_n)\to(\tilde{u},\tilde{v}) \in \mathbb{D}_r$ up to subsequence.
\end{lemma}

\begin{proof}
As in Lemma~\ref{lemmaPS2} and Lemma~\ref{lemmaPS1}, to exclude concentration at the $x=0$, it is enough to prove that
\begin{equation}\label{radial:at0}
\lim\limits_{\varepsilon\to0}\limsup\limits_{n\to+\infty}\int_{\mathbb{R}^N}h(x)u_n^2v_n\varphi_{0,\varepsilon}(x)dx=0,
\end{equation}
for $\varphi_{j,\varepsilon}$ a smooth cut--off function centered at the origin defined as in \eqref{cutoff}. To exclude concentration at $\infty$, it is enough to show that
\begin{equation}\label{radial:atinf}
\lim\limits_{R\to+\infty}\limsup\limits_{n\to+\infty}\int_{|x|>R}h(x)u_n^2v_n\varphi_{\infty,\varepsilon}(x)dx=0,
\end{equation}
where $\varphi_{\infty,\varepsilon}$ is a cut--off function supported near $\infty$, see \eqref{cutoffinfi}. To prove \eqref{radial:at0}, observe that, because of H\"older's inequality,
\begin{equation}\label{ineq:holder}
\begin{split}
\int_{\mathbb{R}^N}h(x)u_n^2v_n\varphi_{0,\varepsilon}(x)dx&\leq\left(\int_{\mathbb{R}^N}h(x)|u_n|^{2^*}\varphi_{0,\varepsilon}dx\right)^{\frac{2}{2^*}}\left(\int_{\mathbb{R}^N}h(x)|v_n|^{2^*}\varphi_{0,\varepsilon}dx\right)^{\frac{1}{2^*}}.
\end{split}
\end{equation}
Hence, by \eqref{con-comp} and \eqref{hypH}, it follows that
\begin{equation*}
\lim\limits_{n\to+\infty}\int_{\mathbb{R}^N}h|u_n|^{2^*}\varphi_{0,\varepsilon}dx=\int_{\mathbb{R}^N}h|\tilde{u}|^{2^*}\varphi_{0,\varepsilon}dx+\rho_0h(0)\leq\int_{|x|\leq\varepsilon}h|\tilde{u}|^{2^*}dx,
\end{equation*}
\begin{equation*}
\lim\limits_{n\to+\infty}\int_{\mathbb{R}^N}h|v_n|^{2^*}\varphi_{0,\varepsilon}dx=\int_{\mathbb{R}^N}h|\tilde{v}|^{2^*}\varphi_{0,\varepsilon}dx+\overline{\rho}_0h(0)\leq\int_{|x|\leq\varepsilon}h|\tilde{v}|^{2^*}dx.
\end{equation*}
Thus, we conclude
\begin{equation*}
\lim\limits_{\varepsilon\to0}\limsup\limits_{n\to+\infty}\int_{\mathbb{R}^N}h(x)u_n^2v_n\varphi_{0,\varepsilon}(x)dx\leq\lim\limits_{\varepsilon\to0}\left(\int_{|x|\leq\varepsilon}h|\tilde{u}|^{2^*}dx\right)^{\frac{2}{2^*}}\left(\int_{|x|\leq\varepsilon}h|\tilde{v}|^{2^*}dx\right)^{\frac{1}{2^*}}=0.
\end{equation*}
Since $\lim\limits_{|x|\to+\infty}h(x)=0$, the proof of \eqref{radial:atinf} follows analogously.
\end{proof}

The PS condition for the non-radial case follows assuming that $\nu$ is small enough.
\begin{lemma}\label{lemcritic2}
Suppose $N=6$ and \eqref{hypH} holds. Let $\{(u_n,v_n)\} \subset \mathbb{D}$ be  a PS sequence for $\mathcal{J}_\nu$ at level $c\in\mathbb{R}$ such that
\begin{equation*}
c<\frac{1}{N} \min\{ \mathcal{S}(\lambda_1),\mathcal{S}(\lambda_2) \}^{\frac{N}{2}}.
\end{equation*}
Then, there exists $\overline{\nu}>0$ such that, for every $\nu\leq \overline{\nu}$,  $(u_n,v_n)\to(\tilde{u},\tilde{v}) \in \mathbb{D}$ up to subsequence.
\end{lemma}

\begin{proof}
Concentration at the points $0$ and $\infty$ can be excluded by similar arguments to those of Lemma~\ref{lemcritic}, so we only have to consider concentration at $x_j\neq0, \infty$. Furthermore, we can also assume that $j\in\mathfrak{J}\cap\mathfrak{K}$. Otherwise, for $\varphi_{j,\varepsilon}(x)$ a cut-off function centered at $x_j\in\mathbb{R}^N$ defined as in \eqref{cutoff} we have
\begin{equation*}
\lim\limits_{\varepsilon\to0}\limsup\limits_{n\to+\infty}\int_{\mathbb{R}^N}h(x)u_n^2v_n\varphi_{j,\varepsilon}(x)dx=0,
\end{equation*}
and, then, there is no concentration at $x_j\in\mathbb{R}^N$ with $j\in\mathfrak{J}$ and $j\notin\mathfrak{K}$ or $x_k\in\mathbb{R}^N$ with $k\notin\mathfrak{J}$ and $k\in\mathfrak{K}$. Therefore, assuming $j\in\mathfrak{J}\cap\mathfrak{K}$ and testing $\mathcal{J}_{\nu}'(u_n,v_n)$ with $(u_n\varphi_{j,\varepsilon},0)$ we get
\begin{equation}\label{limitu}
\begin{split}
0&=\lim\limits_{n\to+\infty}\left\langle \mathcal{J}_{\nu}'(u_n,v_n)\big|(u_n\varphi_{j,\varepsilon},0)\right\rangle\\
&=\lim\limits_{n\to+\infty}\left(\int_{\mathbb{R}^N}|\nabla u_n|^2\varphi_{j,\varepsilon}dx+\int_{\mathbb{R}^N}u_n\nabla u_n\nabla\varphi_{j,\varepsilon}dx-\lambda_1\int_{\mathbb{R}^N}\frac{u_n^2}{|x|^2}\varphi_{j,\varepsilon}dx\right.\\
&\mkern+80mu-\left.\int_{\mathbb{R}^N}|u_n|^{2^*}\varphi_{j,\varepsilon}dx-2\nu\int_{\mathbb{R}^N}h(x)u_n^2v_n{\varphi_{j,\varepsilon}}dx\right),
\end{split}
\end{equation}
and testing $\mathcal{J}_{\nu}'(u_n,v_n)$ with $(0,v_n\varphi_{j,\varepsilon})$ we get
\begin{equation}\label{limitv}
\begin{split}
0&=\lim\limits_{n\to+\infty}\left\langle \mathcal{J}_{\nu}'(u_n,v_n)\big|(0,v_n\varphi_{j,\varepsilon})\right\rangle\\
&=\lim\limits_{n\to+\infty}\left(\int_{\mathbb{R}^N}|\nabla v_n|^2\varphi_{j,\varepsilon}dx+\int_{\mathbb{R}^N}v_n\nabla v_n\nabla\varphi_{j,\varepsilon}dx-\lambda_2\int_{\mathbb{R}^N}\frac{v_n^2}{|x|^2}\varphi_{j,\varepsilon}dx\right.\\
&\mkern+80mu-\left.\int_{\mathbb{R}^N}|v_n|^{2^*}\varphi_{j,\varepsilon}dx-\nu\int_{\mathbb{R}^N}h(x)u_n^2v_n{\varphi_{j,\varepsilon}}dx\right).
\end{split}
\end{equation}
Hence, as $h\in L^{\infty}(\mathbb{R}^N)$, by \eqref{ineq:holder}, we get
\begin{equation}\label{conuv}
\lim\limits_{\varepsilon\to0}\limsup\limits_{n\to+\infty}\int_{\mathbb{R}^N}h(x)u_n^2v_n\varphi_{j,\varepsilon}(x)dx\leq \tilde{C}\rho_j^{\frac{2}{2^*}}\overline{\rho}_j^{\frac{1}{2^*}}.
\end{equation}
Therefore, letting $\varepsilon\to0$, from \eqref{limitu}, \eqref{limitv} and \eqref{conuv} it follows that
\begin{equation*}
\mu_j-\rho_j-2\nu\tilde{C}\rho_j^{\frac{2}{2^*}}\overline{\rho}_j^{\frac{1}{2^*}}\leq0\qquad\text{and}\qquad
\overline{\mu}_j-\overline{\rho}_j-\nu\tilde{C}\rho_j^{\frac{2}{2^*}}\overline{\rho}_j^{\frac{1}{2^*}}\leq0.
\end{equation*}
Thus, because of \eqref{ineq:sobcon}, we get
\begin{equation*}
\mathcal{S}\left(\rho_j^{\frac{2}{2^*}}+\overline{\rho}_j^{\frac{1}{2^*}}\right)\leq\rho_j+\overline{\rho}_j+2^*\nu \tilde{C}\rho_j^{\frac{2}{2^*}}\overline{\rho}_j^{\frac{1}{2^*}},
\end{equation*}
so $\mathcal{S}\left(\rho_j+\overline{\rho}_j\right)^{\frac{2}{2^*}}\leq(\rho_j+\overline{\rho}_j)(1+2^*\nu \tilde{C})$. Then, either $\rho_j+\overline{\rho}_j=0$ or $\displaystyle\rho_j+\overline{\rho}_j\ge\left(\frac{\mathcal{S}}{1+2^*\nu \tilde{C}}\right)^{\frac{N}{2}}$. As in Lemma~\ref{lemmaPS2}, in case of having concentration, we get
\begin{equation*}
c\ge\frac16(\mu_j+\overline{\mu}_j)\ge S\frac16(\rho_j+\overline{\rho}_j)^{\frac{2}{2^*}}
\ge\frac{1}{N}\left(\frac{\mathcal{S}}{1+2^*\nu \tilde{C}}\right)^{\frac{N}{2}}.
\end{equation*}
Hence, for $\nu>0$ sufficiently small, we find
\begin{equation*}
c\ge\frac{1}{N}\left(\frac{\mathcal{S}}{1+2^*\nu \tilde{C}}\right)^{\frac{N}{2}}\ge\frac{1}{N}\min\{\mathcal{S}(\lambda_1),\mathcal{S}(\lambda_2)\}^{\frac{N}{2}},
\end{equation*}
in contradiction with the hypothesis on the energy level $c$.
\end{proof}

\section{Main Results}\label{section4}
We prove now the main theorems regarding the solvability of the system \eqref{system:SKdV}. In this section, we shall assume one of the following
\begin{equation}\tag{C}\label{alternative}
 \mbox{ Either } 3\le N \le 5 \qquad  \mbox{ or } \qquad N=6  \mbox{ and } h  \mbox{ is radial and satisfies \eqref{hypH}, }
\end{equation}
\begin{equation}\tag{D}\label{alternative2}
{ N=6, \: \nu \mbox{ satisfies Lemma~\ref{lemcritic2}} \quad \mbox{and} \quad \eqref{hypH} \: \mbox{holds}.}
\end{equation}

The first result addresses the case $\nu>\overline{\nu}$.  By Proposition~\ref{thm:semitrivial},  {the semi-trivial solution $(0,z_\mu^{\lambda_2})$ is a saddle point of $\mathcal{J}_\nu$ constrained to $\mathcal{N}_\nu$. See Figure~1 for a scheme of this situation.}

\begin{theorem}\label{thm:nugrande}
Assume that $\nu>\overline{\nu}$ defined by \eqref{overnu}. If \eqref{alternative} holds,
 then system \eqref{system:SKdV} admits a positive ground state solution $(\tilde{u},\tilde{v}) \in \mathbb{D}$.  
\end{theorem}

\begin{proof}
By Proposition~\ref{thm:semitrivial}, the couple $(0,z_\mu^{\lambda_2})$ is a  saddle point of $\mathcal{J}_\nu$ constrained on $\mathcal{N}_\nu$. Recall that $(z_\mu^{\lambda_1},0)$ is not a critical point of $\mathcal{J}_\nu$ on $\mathcal{N}_\nu$. Consequently,
\beq\label{minimumlevel}
\tilde{c}_\nu < \min\{\mathcal{J}_\nu(z_\mu^{\lambda_1},0),\mathcal{J}_\nu(0,z_\mu^{\lambda_2}) \}=\frac{1}{N}\min\{\mathcal{S}(\lambda_1),\mathcal{S}(\lambda_2)\}^{\frac{N}{2}},
\eeq
where $\tilde{c}_\nu $ is defined in \eqref{ctilde}. For a subcritical dimension, $3\le N\le 5$, Lemma~\ref{lemmaPS2} guarantees the existence of $(\tilde{u},\tilde{v}) \in \mathbb{D}$ such that $\mathcal{J}_\nu(\tilde{u},\tilde{v})=\tilde{c}_\nu$. In addition, due to
\beq\label{minimumlevel2}
\mathcal{J}_\nu(|\tilde{u}|,|\tilde{v}|)\leq \mathcal{J}_\nu(\tilde{u},\tilde{v}),
\eeq
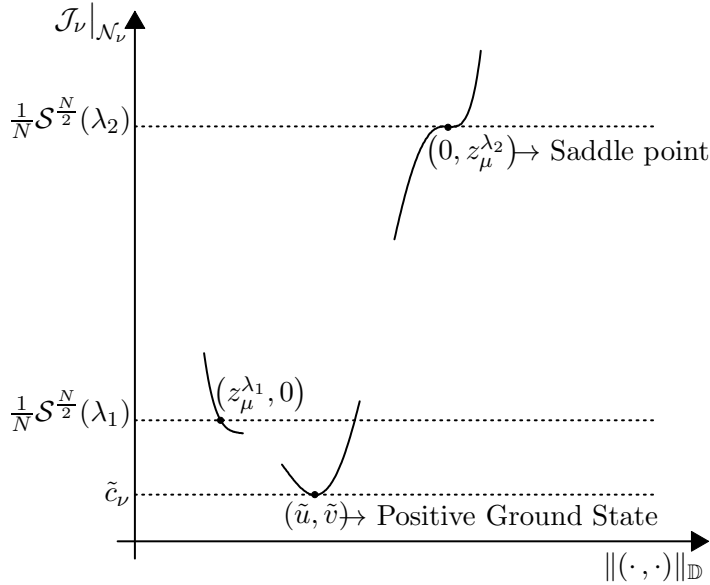
\begin{figure}[b]
\begin{tikzpicture}[scale=1.15,line cap=round,line join=round,>=triangle 45,x=1cm,y=1cm]
\draw [line width=0.75pt] (0,-0.2)-- (0,6);
\draw [line width=0.75pt] (-0.2,0)-- (6.5,0);
\draw [fill=black,shift={(0,6)}] (0,0) ++(0 pt,3.75pt) -- ++(3.2475952641916446pt,-5.625pt)--++(-6.495190528383289pt,0 pt) -- ++(3.2475952641916446pt,5.625pt);
\draw [fill=black,shift={(6.5,0)},rotate=270] (0,0) ++(0 pt,3.75pt) -- ++(3.2475952641916446pt,-5.625pt)--++(-6.495190528383289pt,0 pt) -- ++(3.2475952641916446pt,5.625pt);
\draw[line width=0.75pt,smooth,samples=100,domain=0.8:1.25] plot(\x,{0.6568838293036278*(\x)^(6)-9.274898053426458*(\x)^(5)+51.66536946561714*(\x)^(4)-144.44701525508654*(\x)^(3)+213.94517328424803*(\x)^(2)-160.37490356950588*(\x)+49.213535512197545});
\draw[line width=0.75pt,smooth,samples=100,domain=1.7:2.6] plot(\x,{0.6568838293036278*(\x)^(6)-9.274898053426458*(\x)^(5)+51.66536946561714*(\x)^(4)-144.44701525508654*(\x)^(3)+213.94517328424803*(\x)^(2)-160.37490356950588*(\x)+49.213535512197545});
\draw[line width=0.75pt,smooth,samples=100,domain=3:4] plot(\x,{0.6568838293036278*(\x)^(6)-9.274898053426458*(\x)^(5)+51.66536946561714*(\x)^(4)-144.44701525508654*(\x)^(3)+213.94517328424803*(\x)^(2)-160.37490356950588*(\x)+49.213535512197545});
\draw [line width=0.75pt,dotted] (0,4.8)-- (6,4.8);
\draw [fill=black] (3.62,4.79) circle (1pt);
\draw[color=black] (3.9,4.5) node {$\left(0,z_{\mu}^{\lambda_2}\right)$};
\draw[color=black] (5.5,4.5) node {$\rightarrow$ Saddle point};
\draw[color=black] (-0.75,4.9) node {$\frac{1}{N}\mathcal{S}^{\frac{N}{2}}(\lambda_2)$};
\draw [line width=0.75pt,dotted] (0,1.4)-- (6,1.4);
\draw [fill=black] (0.99,1.4) circle (1pt);
\draw[color=black] (1.45,1.7) node {$\left(z_{\mu}^{\lambda_1},0\right)$};
\draw[color=black] (-0.75,1.51) node {$\frac{1}{N}\mathcal{S}^{\frac{N}{2}}(\lambda_1)$};
\draw [line width=0.75pt,dotted] (0,0.54)-- (6,0.54);
\draw [fill=black] (2.08,0.54) circle (1pt);
\draw[color=black] (2.1,0.3) node {$(\tilde{u},\tilde{v})$};
\draw[color=black] (4.2,0.3) node {$\rightarrow$ Positive Ground State};
\draw[color=black] (-0.2,0.55) node {$\tilde{c}_{\nu}$};
\draw[color=black] (6,-0.3) node {$\|(\cdot \, ,\cdot)\|_{\mathbb{D}}$};
\draw[color=black] (-0.5,6) node {$\mathcal{J}_{\nu}{\big|}_{\mathcal{N}_{\nu}}$};

\end{tikzpicture}
\caption{The energy configuration under hypotheses of Theorem~\ref{thm:nugrande}}
\end{figure}
we can assume that $\tilde{u}\ge 0$ and $\tilde{v}\ge 0$ in $\mathbb{R}^N$. By classical regularity results, $\tilde{u}$ and $\tilde{v}$ are smooth in $\R^N\setminus\{0\}$. Moreover, $\tilde{u}\not \equiv 0$ and $\tilde{v}\not \equiv 0$. Otherwise, if $\tilde u\equiv 0$, one obtains that $\tilde{v}$ satisfies \eqref{vtildeq}. Actually, $\tilde{v}=z_\mu^{\lambda_2}$, which violates \eqref{minimumlevel}. The case $\tilde v\not\equiv 0$, can not take place since, on the contrary, both $\tilde{u},\tilde{v}\equiv 0$ and $(0,0)\not\in\mathcal{N}_\nu$. Finally, using the maximum principle in $\R^N\setminus\{0\}$, one derives that $(\tilde{u},\tilde{v}) \in \mathcal{N}_\nu$ is a ground state such that $\tilde{u}> 0$ and $\tilde{v}> 0$ in $\mathbb{R}^N\setminus\{0\}$. The same conclusion holds for the critical dimension $N=6$, by applying Lemma~\ref{lemcritic} instead. Consequently, also we infer that $(\tilde{u},\tilde{v})$ is a positive ground state.
\end{proof}

We point out that the order between the energy levels of the semi-trivial solution and $(z_\mu^{\lambda_1},0)$ is determined by the order of the parameters $\lambda_1$ and $\lambda_2$, since \eqref{Jzeta} and \eqref{Slambda} illustrate. Indeed, if $\lambda_1\ge \lambda_2$, the minimum level between both corresponds to $(z_\mu^{\lambda_1},0)$, which is not a critical point of $\mathcal{J}_\nu$ on $\mathcal{N}_\nu$. As an immediate consequence, the existence of a positive ground state is derived. See Figure~2 for the corresponding energy configuration. Note that, in this figure, $(0,z_\mu^{\lambda_2})$ is assumed to be a local minimum, but it may be a saddle point.
\begin{theorem}\label{thm:lambdaground} Suppose $\lambda_1\ge \lambda_2$. If either \eqref{alternative} or \eqref{alternative2} holds, then system \eqref{system:SKdV} admits a positive ground state $(\tilde{u},\tilde{v})\in\mathbb{D}$.
\end{theorem}
\begin{proof}
Since $\lambda_1\ge \lambda_2$ and { $(z_\mu^{\lambda_1},0)$ is not a critical point of $\mathcal{J}_\nu$ constrained on $\mathcal{N}_\nu$},
$$
\tilde{c}_\nu<\mathcal{J}_\nu(z_\mu^{\lambda_1},0)=\frac{1}{N}\mathcal{S}^{\frac{N}{2}}(\lambda_1)=\frac{1}{N}\min \{\mathcal{S}(\lambda_1),\mathcal{S}(\lambda_2)\}^{\frac{N}{2}},
$$
with $\tilde{c}_\nu $ was introduced in \eqref{ctilde}. Therefore, for subcritical dimension, $3\le N\le 5$, Lemma~\ref{lemmaPS2} implies that there exists $(\tilde{u},\tilde{v})\in\mathcal{N}_\nu$ with $\tilde{c}_\nu=\mathcal{J}_\nu(\tilde{u},\tilde{v})$. Using \eqref{minimumlevel2}, one can suppose that $u,v\ge 0$ in $\mathbb{R}^N$. Moreover, arguing by contradiction, it is deduced easily that $(\tilde{u},\tilde{v}) \not \equiv (0,0)$. Applying the maximum principle in $\R^N\setminus\{0\}$, we obtain the desired conclusion.\newline
For the case of critical dimension $N=6$, we arrive at the existence of a positive ground state $(\tilde{u},\tilde{v})$ of  \eqref{system:SKdV}, by using Lemma~\ref{lemcritic} instead. On the other hand, for $\nu>0$ small enough, Lemma~\ref{lemcritic2} provides the conclusion. 
\end{proof}


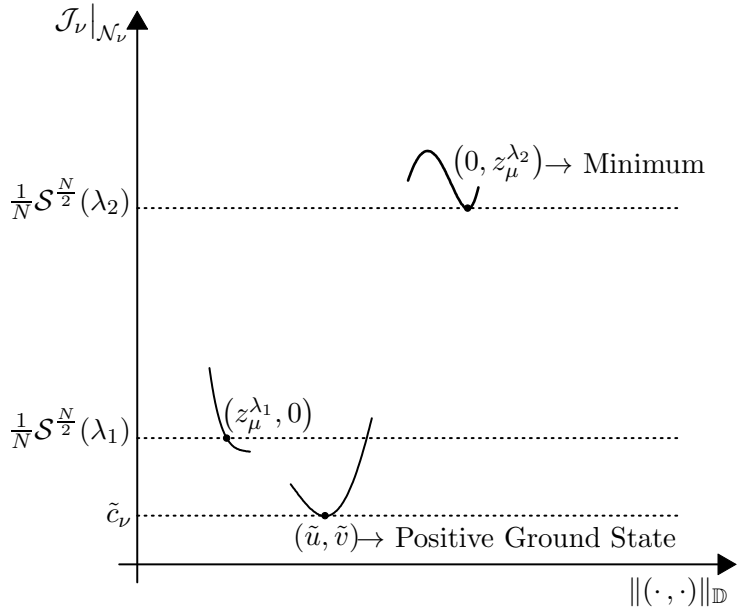
\begin{figure}[b]
\begin{tikzpicture}[scale=1.2,line cap=round,line join=round,>=triangle 45,x=1cm,y=1cm]
\draw [line width=0.75pt] (0,-0.2)-- (0,6);
\draw [line width=0.75pt] (-0.2,0)-- (6.5,0);
\draw [fill=black,shift={(0,6)}] (0,0) ++(0 pt,3.75pt) -- ++(3.2475952641916446pt,-5.625pt)--++(-6.495190528383289pt,0 pt) -- ++(3.2475952641916446pt,5.625pt);
\draw [fill=black,shift={(6.5,0)},rotate=270] (0,0) ++(0 pt,3.75pt) -- ++(3.2475952641916446pt,-5.625pt)--++(-6.495190528383289pt,0 pt) -- ++(3.2475952641916446pt,5.625pt);
\draw[line width=0.75pt,smooth,samples=100,domain=0.8:1.25] plot(\x,{0.6568838293036278*(\x)^(6)-9.274898053426458*(\x)^(5)+51.66536946561714*(\x)^(4)-144.44701525508654*(\x)^(3)+213.94517328424803*(\x)^(2)-160.37490356950588*(\x)+49.213535512197545});
\draw[line width=0.75pt,smooth,samples=100,domain=1.7:2.6] plot(\x,{0.6568838293036278*(\x)^(6)-9.274898053426458*(\x)^(5)+51.66536946561714*(\x)^(4)-144.44701525508654*(\x)^(3)+213.94517328424803*(\x)^(2)-160.37490356950588*(\x)+49.213535512197545});
\draw[line width=1pt,smooth,samples=100,domain=3:3.78] plot(\x,{0.7687058160669185*(\x)^(6)-9.837471556437674*(\x)^(5)+48.667007330076935*(\x)^(4)-117.3768839399743*(\x)^(3)+144.26776095512017*(\x)^(2)-86.0790448192908*(\x)+21.350550835980123});

\draw [line width=0.75pt,dotted] (0,3.95)-- (6,3.95);
\draw [fill=black] (3.66,3.95) circle (1pt);

\draw[color=black] (4,4.45) node {$\left(0,z_{\mu}^{\lambda_2}\right)$};
\draw[color=black] (5.4,4.45) node {$\rightarrow$ Minimum};
\draw[color=black] (-0.75,4.05) node {$\frac{1}{N}\mathcal{S}^{\frac{N}{2}}(\lambda_2)$};
\draw [line width=0.75pt,dotted] (0,1.4)-- (6,1.4);
\draw [fill=black] (0.99,1.4) circle (1pt);
\draw[color=black] (1.45,1.65) node {$\left(z_{\mu}^{\lambda_1},0\right)$};
\draw[color=black] (-0.75,1.51) node {$\frac{1}{N}\mathcal{S}^{\frac{N}{2}}(\lambda_1)$};
\draw [line width=0.75pt,dotted] (0,0.54)-- (6,0.54);
\draw [fill=black] (2.08,0.54) circle (1pt);
\draw[color=black] (2.1,0.3) node {$(\tilde{u},\tilde{v})$};
\draw[color=black] (4.2,0.3) node {$\rightarrow$ Positive Ground State};
\draw[color=black] (-0.2,0.55) node {$\tilde{c}_{\nu}$};
\draw[color=black] (6,-0.3) node {$\|(\cdot \, ,\cdot)\|_{\mathbb{D}}$};
\draw[color=black] (-0.5,6) node {$\mathcal{J}_{\nu}{\big|}_{\mathcal{N}_{\nu}}$};
\end{tikzpicture}
\caption{The energy configuration under hypotheses of Theorem~\ref{thm:lambdaground}}
\end{figure}


Next, we focus on the case that $0<\nu<\overline{\nu}$. In the following result, we infer that if the minimum energy level of the semi-trivial couples corresponds to the semi-trivial solution $(0,z_\mu^{\lambda_2})$, i.e. $\lambda_2>\lambda_1$, it is indeed a ground state to \eqref{system:SKdV} for $\nu$ sufficiently small.

\begin{theorem}\label{thm:groundstates}
Assume $\lambda_2>\lambda_1$. If either \eqref{alternative} or \eqref{alternative2} holds, then there exists $\tilde{\nu}>0$ such that for any $0<\nu<\tilde{\nu}$  the couples $(0,\pm z_\mu^{\lambda_2})$ are critical points of minimal energy for $\mathcal{J}_\nu$ on $\mathcal{N}_\nu$.
Even more, $(0,z_\mu^{\lambda_2})$ is a ground state to \eqref{system:SKdV}.
\end{theorem}

\begin{proof}
Let us suppose by contradiction that there exists a sequence $\nu_n \searrow 0$ whose energy level satisfies  $\tilde{c}_{\nu_n} <  \mathcal{J}_{\nu_n} (0,z_\mu^{\lambda_2})$, where $\tilde{c}_{\nu_n}$ defined in \eqref{ctilde} with $\nu=\nu_n$. Moreover, by the assumption $\lambda_2>\lambda_1$, we have
\beq\label{groundstates1}
\tilde{c}_{\nu_n}< \frac{1}{N} \min \{ \mathcal{S}(\lambda_1), \mathcal{S}(\lambda_2) \}^{\frac{N}{2}}= \frac{1}{N} \mathcal{S}^{\frac{N}{2}}(\lambda_2).
\eeq
If $3\le N\le 5$, the PS condition is satisfied at level $\tilde{c}_{\nu_n}$, thanks to Lemma~\ref{lemmaPS2}. If $N=6$, the compactness follows from Lemmas~\ref{lemcritic} and \ref{lemcritic2}.  Thus, we derive the existence of $(\tilde{u}_n,\tilde{v}_n) \in \mathbb{D}$ with $\tilde{c}_{\nu_n}=\mathcal{J}_{\nu_n} (\tilde{u}_n,\tilde{v}_n)$. By \eqref{minimumlevel2}, one can suppose that $\tilde{u}_n \ge 0$ and $\tilde{v}_n \ge 0$. Furthermore, by contradiction, we infer that $\tilde{u}_n \not \equiv 0$ and $\tilde{v}_n \not \equiv 0$ in $\mathbb{R}^N$. Finally, one can conclude that $\tilde{u}_n > 0$ and $\tilde{v}_n > 0$ in $\mathbb{R}^N\setminus \{0\}$ by applying the maximum principle.\newline 
Let us define
\begin{equation*}
\sigma_{1,n}=\int_{\mathbb{R}^N} \tilde{u}_n^{2^*} \, dx \qquad \mbox{ and } \qquad \sigma_{2,n}=\int_{\mathbb{R}^N} \tilde{v}_n^{2^*} \, dx .
\end{equation*}
By \eqref{Nnueq2}, one obtains
\beq\label{groundstates3}
\tilde{c}_{\nu_n} = \mathcal{J}_{\nu_n} (\tilde{u}_n,\tilde{v}_n) =  \frac{1}{N} \left(  \sigma_{1,n} + \sigma_{2,n}\right) + \frac{\nu_n}{2}  \int_{\mathbb{R}^N} h(x)  \,  \tilde{u}_n^{2} \,  \tilde{v}_n \, dx  .
\eeq
Combining \eqref{groundstates1} and \eqref{groundstates3}, we deduce that
\beq\label{groundstates4}
\sigma_{1,n}+\sigma_{2,n}<\mathcal{S}^{\frac{N}{2}}(\lambda_2).
\eeq

Now use that $\tilde{u}_n$ and $\tilde{v}_n$ satisfy \eqref{system:SKdV}. From the first equation of \eqref{system:SKdV} and \eqref{Slambda}, we get
\beq\label{groundstates45}
\mathcal{S}(\lambda_1) (\sigma_{1,n})^{\frac{N-2}{N}} \leq \sigma_{1,n} + 2 \nu_n  \int_{\mathbb{R}^N} h(x)  \,  \tilde{u}_n^{2} \,  \tilde{v}_n \, dx  .
\eeq
Hence, applying H\"older's inequality and \eqref{groundstates4}, it follows that
$$
\int_{\mathbb{R}^N} h(x)  \,  \tilde{u}_n^{2} \,  \tilde{v}_n \, dx    \leq \|h\|_{L^\infty} \left( \int_{\mathbb{R}^N} \tilde{u}_n^{2^*}  \, dx   \right)^{\frac{2}{2^*}}\left( \int_{\mathbb{R}^N} \tilde{v}_n^{2^*}  \, dx    \right)^{\frac{1}{2^*}} \leq { \|h\|_{L^\infty} (\mathcal{S}(\lambda_2))^{\frac{N-2}{4}} (\sigma_{1,n})^{\frac{N-2}{N}}.}
$$
 Introducing the above inequality in \eqref{groundstates45}, it follows that
\begin{equation*}
\mathcal{S}(\lambda_1) (\sigma_{1,n})^{\frac{N-2}{N}} < \sigma_{1,n} + 2\nu_n   C(h) (\mathcal{S}(\lambda_2))^{\frac{N-2}{4}} (\sigma_{1,n})^{\frac{N-2}{N}}.
\end{equation*}
As $\lambda_2> \lambda_1$, there exists $\varepsilon>0$ such that
\beq\label{groundstates6}
(1-\varepsilon) \mathcal{S}^{\frac{N}{2}}(\lambda_1) \ge \mathcal{S}^{\frac{N}{2}}(\lambda_2).
\eeq
Next, we apply Lemma~\ref{algelemma} to $\sigma_{1,n}$ and we deduce the existence of $\tilde{\nu}=\tilde{\nu}(\varepsilon)>0$ such that
$$
\sigma_{1,n}> (1-\varepsilon) \mathcal{S}^{\frac{N}{2}}(\lambda_1) \qquad \mbox{ for any } 0<\nu_n<\tilde{\nu}.
$$
Since parameter $\varepsilon$ satisfies \eqref{groundstates6}, it follows that $\sigma_{1,n}>\mathcal{S}^{\frac{N}{2}}(\lambda_2)$, in contradiction with \eqref{groundstates4}. Thus, for $\nu$ small enough,
$$
\tilde{c}_\nu = { \frac{1}{N} \mathcal{S}^{\frac{N}{2}}(\lambda_2).}
$$
If $(\tilde{u},\tilde{v})$ is a minimizer of $\mathcal{J}_\nu$, repeating the above argument, it follows that $\tilde{u}\equiv 0$. In addition, $\tilde{v}$ solves to
$$
-\Delta \tilde{v} - \lambda_2 \frac{\tilde{v}}{|x|^2}=|\tilde{v}|^{2^*-2}\tilde{v} \qquad \mbox{ in } \mathbb{R}^N.
$$
We prove now that $\tilde{v}$ does not change its sign and, actually, $\tilde{v}= \pm z_{\mu}^{\lambda_2}$. Arguing by contradiction, we shall suppose that $\tilde{v}$ is sign-changing. Then, $\tilde{v}^{\pm} \not \equiv 0$ in $\mathbb{R}^N$. Due to $(0,\tilde{v}) \in \mathcal{N}_\nu$, one obtains $(0,\tilde{v}^\pm) \in \mathcal{N}_\nu$. By using the equality \eqref{groundstates3}, one gets
$$
\tilde{c}_{\nu}= \mathcal{J}_\nu (0,\tilde{v}) = \frac{1}{N} \int_{\mathbb{R}^N} { |\tilde{v}|}^{2^*}  \, dx  = \frac{1}{N} \left( \int_{\mathbb{R}^N} (\tilde{v}^+)^{2^*}  \, dx  + \int_{\mathbb{R}^N}  {|\tilde{v}^-|}^{2^*}  \, dx  \right) > \mathcal{J}_\nu (0,\tilde{v}^+) \ge  \tilde{c}_{\nu},
$$
contradicting the fact that the energy of $(0,\tilde{v})$ is minimum. Hence, $(0,\pm z_{\mu}^{\lambda_2})$ is the minimizer of $\mathcal{J}_\nu$ in $\mathcal{N}_\nu$ if $\lambda_2>\lambda_1$. Furthermore, the ground state to \eqref{system:SKdV} corresponds to $(0,z_{\mu}^{\lambda_2})$.
\end{proof}
\begin{remark}
If $\lambda_2-\lambda_1$ increases, the interval for admissible $\nu$ in Theorem~\ref{thm:groundstates} increases. Indeed, the greater the difference $\lambda_2-\lambda_1$, the greater the range of $\varepsilon$ whose satisfies \eqref{groundstates6}. Consequently, we can choose a bigger $\tilde{\nu}$ in Lemma~\ref{algelemma}.
\end{remark}

Finally, we deduce the existence of bound states by applying a min-max argument. In particular, it is proved that the energy functional $\mathcal{J}^+_\nu$, presented in \eqref{funct:SKdVp}, exhibits the \textit{Mountain-Pass} geometry for certain choice of parameters $\lambda_1,\lambda_2$. This assumption, a kind of separability condition, allows us to establish a proper separation between the semi-trivial energy levels. In Figure~3, we can see the couple $(0,z_{\mu}^{\lambda_2})$ as a ground state, provided by Theorem~\ref{thm:groundstates}, and the bound state provided by the following theorem.

\begin{theorem}\label{thm:MPgeom}
Assume that $\lambda_2> \lambda_1$ and
\beq\label{lamdas}
2^{-\frac{2}{N-1}}<\frac{\Lambda_N-\lambda_2}{\Lambda_N-\lambda_1}.
\eeq
If \eqref{alternative} holds, then there exists  $\tilde{\nu}>0$ such that, for $0<\nu\le\tilde{\nu}$, $\mathcal{J}^+_\nu\Big|_{\mathcal{N}^+_\nu}$ admits a \it{Mountain-Pass} critical point $(\tilde{u},\tilde{v})\in\mathbb{D}$ which is a positive bound state to \eqref{system:SKdV}.
\end{theorem}

\begin{proof}
 The proof is divided into two steps. In the first one, we prove that the energy functional $\mathcal{J}_\nu^+$ admits the { \it Mountain-pass} geometry, whereas in the second one we prove that for the {\it Mountain-pass} level the PS condition is guaranteed. As a consequence, there exists a critical point $(\tilde{u},\tilde{v})\in\mathbb{D}$ of $\mathcal{J}_\nu^+$.

First, let us define the set of paths connecting $(z_{\mu}^{\lambda_1},0)$ with $(0,z_{\mu}^{\lambda_2})$ continuously,
$$
\Psi_\nu = \left\{ \psi=(\psi_1,\psi_2)\in C^0([0,1],\mathcal{N}^+_\nu), \quad  \, \psi(0)=(z_1^{\lambda_1},0) \, \mbox{ s. t. } \mbox{ and } \, \psi(1)=(0,z_1^{\lambda_2})\right\},
$$
and the MP level
\begin{equation*}
c_{MP} = \inf_{\psi\in\Psi_\nu} \max_{t\in [0,1]} \mathcal{J}^+_{\nu} (\psi(t)).
\end{equation*}
The hypothesis \eqref{lamdas} implies that
$$
\frac{2}{N} \mathcal{S}^{\frac N2}(\lambda_2) > \frac{1}{N} \mathcal{S}^{\frac N2}(\lambda_1).
$$
Due to the continuity and monotonicity of $\mathcal{S}(\lambda)$, one can fix $\varepsilon>0$ small enough with
\beq\label{MPgeomp0}
\frac{2}{N}(1-\varepsilon)\left( \frac{\mathcal{S}(\lambda_1)+\mathcal{S}(\lambda_2)}{2}\right)^{\frac{N}{2}}> \frac{2}{N} \mathcal{S}^{\frac{N}{2}}(\lambda_2)>\frac{1+\varepsilon}{N} \mathcal{S}^{\frac{N}{2}}(\lambda_1).
\eeq

\

\textbf{Claim}: There exists $\tilde{\nu}=\tilde{\nu}(\varepsilon)>0$ such that, for any $0<\nu<\tilde{\nu}$, we have
\beq\label{claim1}
\max_{t\in[0,1]} \mathcal{J}_\nu^+ (\psi(t)) \ge \frac{2}{N}(1-\varepsilon)\left( \frac{\mathcal{S}(\lambda_1)+\mathcal{S}(\lambda_2)}{2}\right)^{\frac{N}{2}} \qquad \mbox{ with } \psi \in \Psi_\nu.
\eeq
Taking $\psi=(\psi_1,\psi_2) \in \Psi_\nu$, and applying \eqref{Nnueq1} to $\mathcal{J}^+_\nu$, we obtain that
\begin{align}\label{MPgeomp1}
 \int_{\mathbb{R}^N} \left( |\nabla \psi_1(t)|^2 + |\nabla \psi_2(t)|^2  \right) \, dx -\lambda_1 \int_{\mathbb{R}^N} \dfrac{\psi_1^2(t)}{|x|^2}  \, dx   - \lambda_2 \int_{\mathbb{R}^N} \dfrac{\psi_2^2(t)}{|x|^2} \, dx  \vspace{0.3cm}\\
= \int_{\mathbb{R}^N} \left( (\psi_1^+(t))^{2^*} + (\psi_2^+(t))^{2^*}  \right) \, dx  +3\nu \int_{\mathbb{R}^N} h(x) (\psi_1^+(t))^2 {\psi_2(t)} \, dx  \nonumber,
\end{align}
and, by \eqref{Nnueq} applied to $\mathcal{J}^+_\nu$,
\beq\label{MPgeomp2}
\mathcal{J}^+_{\nu} (\psi(t)) = \frac{1}{N} \left( \int_{\mathbb{R}^N} (\psi_1^+(t))^{2^*} + (\psi_2^+(t))^{2^*} \, dx \right)  + \frac{\nu}{2} \int_{\mathbb{R}^N} h(x)  \,  (\psi_1^+(t))^{2} \,  { \psi_2(t) } \, dx.
\eeq
Let us define $\sigma(t)=\left(\sigma_1(t),\sigma_2(t)\right)$ where $\displaystyle\sigma_i(t)=\int_{\mathbb{R}^N} (\psi_i^+(t))^{2^*} \, dx$ for $i=1,2$ and let us assume that $\sigma_i(t)\leq 2 \mathcal{S}^{\frac{N}{2}}(\lambda_1)$ for $t\in[0,1]$ and $i=1,2$ since, on the contrary, \eqref{claim1} is done.\newline
By using the definition of $\mathcal{S}(\lambda)$, we can pass from \eqref{MPgeomp1} to the inequality
\begin{equation}\label{MPgeomp3}
\begin{split}
\mathcal{S}(\lambda_1)(\sigma_1(t))^{\frac{N-2}{N}}+\mathcal{S}(\lambda_2)(\sigma_2(t))^{\frac{N-2}{N}}
\leq & \int_{\mathbb{R}^N} \left( |\nabla \psi_1(t)|^2 + |\nabla \psi_2(t)|^2  \right) \, dx \\
& -\lambda_1 \int_{\mathbb{R}^N} \dfrac{\psi_1^2(t)}{|x|^2} \, dx  - \lambda_2 \int_{\mathbb{R}^N} \dfrac{\psi_2^2(t)}{|x|^2} \, dx\\
= &\: \sigma_1(t)+\sigma_2(t)+3\nu \int_{\mathbb{R}^N} h(x) (\psi_1^+(t))^2 {\psi_2(t)} \, dx.
 \end{split}
 \end{equation}
Moreover, by H\"older's inequality,
\begin{equation}\label{MPgeomp4}
\int_{\mathbb{R}^N} h(x) (\psi_1^+(t))^{2} (\psi_2(t)) \, dx \leq \nu \|h\|_{L^{\infty}(\R^N)} (\sigma_1(t))^{\frac{N-2}{N}} (\sigma_2(t))^{\frac{N-2}{2N}}.
\end{equation}
and by the definition of $\psi$,
$$
\sigma(0)=\left(\int_{\mathbb{R}^N} (z_1^{\lambda_1})^{2^*} \, dx,0\right) \quad \mbox{ and } \quad \sigma(1)=\left(0,\int_{\mathbb{R}^N} (z_1^{\lambda_2})^{2^*} \, dx \right).$$
Since $\sigma$ is continuous, there exists $\tilde{t}\in(0,1)$ with $\sigma_1(\tilde{t})=\tilde{\sigma}=\sigma_2(\tilde{t})$. Combining \eqref{MPgeomp3} with $t=\tilde{t}$ and \eqref{MPgeomp4},
$$
(\mathcal{S}(\lambda_1)+\mathcal{S}(\lambda_2)) \tilde{\sigma}^{\frac{N-2}{N}} \leq 2 \tilde{\sigma} + 3 \nu \tilde{\sigma}^{\frac{3}{2}\frac{N-2}{N}}.
$$
On the other hand, by Lemma~\ref{algelemma}, there exists $\tilde{\nu}$ depending on $\varepsilon$ such that
\begin{equation}\label{MPgeomp5}
\tilde{\sigma}\ge (1-\varepsilon)  \left(\frac{\mathcal{S}(\lambda_1)+\mathcal{S}(\lambda_2)}{2}\right)^{\frac{N}{2}} \quad \mbox{ for every } 0<\nu\le\tilde{\nu}.
\end{equation}
Then, by \eqref{MPgeomp2} and \eqref{MPgeomp5}, one has 
$$
\max_{t\in[0,1]} \mathcal{J}^+_\nu(\psi(t)) \ge \frac{ \sigma_1(t)+ \sigma_2(t)}{N}  \ge {\frac{2(1-\varepsilon) }{N}} \left(\frac{\mathcal{S}(\lambda_1)+\mathcal{S}(\lambda_2)}{2}\right)^{\frac{N}{2}},
$$
proving the claim \eqref{claim1}. In addition, because of \eqref{MPgeomp0} and \eqref{claim1},  we get
\begin{equation}\label{MPgeomp6}
{c_{MP}}>\frac{(1+\varepsilon)}{N} \mathcal{S}^{\frac{N}{2}}(\lambda_1)={(1+\varepsilon)}\mathcal{J}^+_\nu(z_1^{\lambda_1},0).
\eeq
Consequently, the energy functional $\mathcal{J}^+_\nu$ has a {\it Mountain-Pass} geometry on $\mathcal{N}_\nu$.

Now we address the second step. To do so, let us consider
$$
\psi(t) =(\psi_1(t),\psi_2(t))=\left((1-t)^{1/2} z_1^{\lambda_1},t^{1/2}z_1^{\lambda_2} \right)\mbox{ for } t\in[0,1].
$$
Because of the properties of the Nehari manifold $\mathcal{N}^+_\nu$, we can deduce the existence of a positive function $\gamma:[0,1]\to(0,+\infty)$ with the $\gamma \psi \in \mathcal{N}_\nu^+$ for { $t\in[0,1]$}. We point out that $\gamma(0)=\gamma(1)=1$. As above, let us define the integral vector
$$
\sigma(t)=(\sigma_1(t),\sigma_2(t))=\left(\int_{\mathbb{R}^N} \left( \gamma \psi_1(t)\right)^{2^*} \, dx ,
\int_{\mathbb{R}^N}\left( \gamma \psi_2(t)\right)^{2^*} \, dx \right).
$$
Since $z_1^{\lambda_1}\in\mathcal{N}_1$ and  $z_2^{\lambda_1}\in\mathcal{N}_2$, { introduced in \eqref{Nnui}}, it holds
$$
\sigma_1(0)=\|z_1^{\lambda_1}\|^2_{\lambda_1}=\int_{\mathbb{R}^N} (z_1^{\lambda_1})^{2^*} = \mathcal{S}(\lambda_1), \quad \mbox{ and }\quad \sigma_2(1)=\|z_1^{\lambda_2}\|^2_{\lambda_2}=\int_{\mathbb{R}^N} (z_1^{\lambda_2})^{2^*} = \mathcal{S}(\lambda_2).
$$
Since $\gamma\psi(t)\in\mathcal{N}^+_\nu$ and \eqref{normH}, one has that
\begin{align*}
\left\|\left((1-t)^{1/2} z_1^{\lambda_1},t^{1/2}z_1^{\lambda_2} \right)\right\|^2_\mathbb{D}=& \gamma^{2^*-2}(t) \left((1-t)^{2^*/2} \sigma_1(0) + t^{2^*/2} \sigma_2(1) \right) \vspace{0.3cm} \\
& + 3 \nu \gamma (t)(1-t)t^{1/2} \int_{\mathbb{R}^N} h(x) (z_1^{\lambda_1})^2 z_1^{\lambda_2} \, dx.
\end{align*}
By the expression above, we can get an upper bound for the function $\gamma$ as follows,
\beq\label{gammabound}
\gamma^{2^*-2}(t)< \dfrac{\left|\left|\left( \psi_1(t), \psi_2(t) \right)\right|\right|^2_\mathbb{D}}{\int_{\mathbb{R}^N} (\psi_1(t))^{2^*}+(\psi_2(t))^{2^*} \, dx}= \dfrac{(1-t) \sigma_1(0) + t \sigma_2(1)}{(1-t)^{2^*/2} \sigma_1(0) + t^{2^*/2} \sigma_2(1)},
\eeq
for every $t\in(0,1)$. By the definition of $\gamma$, \eqref{gammabound} and \eqref{Nnueq}, one gets
\begin{align}\label{Jgammapsibound}
\mathcal{J}_\nu^+(\gamma \psi(t)) & = \frac{1}{6} \|\gamma\psi(t) \|^2_\mathbb{D}+ \frac{6-N}{6N} \gamma^{2^*}(t) \left(\int_{\mathbb{R}^N} (\psi_1(t))^{2^*}+(\psi_2(t))^{2^*}  \, dx \right) \vspace{0.3cm} \nonumber \\
& = \frac{\gamma^2(t)}{6} \left[(1-t) \sigma_1(0) + t \sigma_2(1) \right] + \frac{6-N}{6N} \gamma^{2^*}(t) \left[ (1-t)^{2^*/2} \sigma_1(0) + t^{2^*/2} \sigma_2(1)\right] \vspace{0.3cm} \\
& < \frac{\gamma^2(t)}{N} \left[(1-t) \sigma_1(0) + t \sigma_2(1) \right] \nonumber.
\end{align}
From \eqref{gammabound}, we have that

$$
\gamma^{2}(t)< \left[ \dfrac{(1-t) \sigma_1(0) + t \sigma_2(1)}{(1-t)^{2^*/2} \sigma_1(0) + t^{2^*/2} \sigma_2(1)} \right]^{\frac{N-2}{2}},
$$
so that, because of \eqref{Jgammapsibound}, for $0<t<1$ we have
$$
\mathcal{J}_\nu^+(\gamma \psi(t)) <  \dfrac{(1-t) \sigma_1(0) + t \sigma_2(1)}{N}  \left[ \dfrac{(1-t) \sigma_1(0) + t \sigma_2(1)}{(1-t)^{2^*/2} \sigma_1(0) + t^{2^*/2} \sigma_2(1)} \right]^{\frac{N-2}{2}}=g(t).
$$
Note that $g(t)$ attains its maximum at $t=\frac{1}{2}$ and
$$
g\left(\frac{1}{2}\right)= \dfrac{ \sigma_1(0) + \sigma_2(1)}{N}=\dfrac{\mathcal{S}^{\frac{N}{2}}(\lambda_1)+\mathcal{S}^{\frac{N}{2}}(\lambda_2)}{N}.
$$
Hence, we have established an upper bound for the {\it Mountain-pass} level $c_{MP}$. More precisely,
$$
{c_{MP} }\leq \max_{t\in[0,1]} \mathcal{J}_\nu^+(\gamma \psi(t))< \dfrac{\mathcal{S}^{\frac{N}{2}}(\lambda_1)+\mathcal{S}^{\frac{N}{2}}(\lambda_2)}{N}.
$$
Finally, introducing the separability condition, by  \eqref{lamdas} and {\eqref{MPgeomp6}}, then
$$
\frac{\mathcal{S}^{\frac{N}{2}}(\lambda_2)}{N}<\frac{\mathcal{S}^{\frac{N}{2}}(\lambda_1)}{N}<c_{MP}<\frac{1}{N}\left( \mathcal{S}^{\frac{N}{2}}(\lambda_1)+\mathcal{S}^{\frac{N}{2}}(\lambda_2) \right)<3\frac{\mathcal{S}^{\frac{N}{2}}(\lambda_2)}{N}
$$
if $\lambda_2>\lambda_1$. The previous inequality means that $c_{MP}$ satisfies the hypotheses of Lemma~\ref{lemmaPS1}. Next, by the {\it Mountain-Pass} Theorem, there exists a sequence $\left\{ (u_n,v_n) \right\} \subset \mathcal{N}^+_\nu$ such that
$$\mathcal{J}^+(u_n,v_n ) \to c_\nu\qquad \displaystyle\mathcal{J}^+|_{\mathcal{N}^+_\nu}(u_n,v_n ) \to 0.$$
Moreover, by Lemma~\ref{lemmaPS1}, $(u_n,v_n) \to ( \tilde{u},\tilde{v} )$. Indeed, $(\tilde{u},\tilde{v} )$ is a critical point  of $\mathcal{J}_\nu$ on $\mathcal{N}_\nu$. Even more, $\tilde{u},\tilde{v} \ge 0$ in $\mathbb{R}^N$. Moreover, the ground state is actually strictly positive by applying maximum principle in $\mathbb{R}^N\setminus\{ 0\}$. We obtain the same conclusion for $N=6$, using Lemma~\ref{lemcritic} for convergence of the PS sequence. 
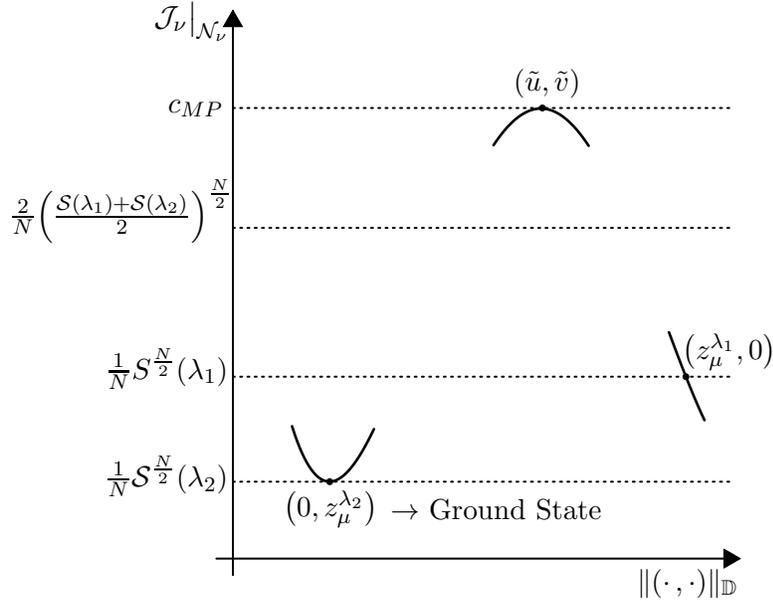
\begin{figure}[b]
\begin{tikzpicture}[scale=1.1, line cap=round,line join=round,>=triangle 45,x=1cm,y=1cm]
\draw [line width=0.75pt] (0,-0.2)-- (0,6.5);
\draw [line width=0.75pt] (-0.2,0)-- (6,0);
\draw [fill=black,shift={(0,6.5)}] (0,0) ++(0 pt,3.75pt) -- ++(3.2475952641916446pt,-5.625pt)--++(-6.495190528383289pt,0 pt) -- ++(3.2475952641916446pt,5.625pt);
\draw [fill=black,shift={(6,0)},rotate=270] (0,0) ++(0 pt,3.75pt) -- ++(3.2475952641916446pt,-5.625pt)--++(-6.495190528383289pt,0 pt) -- ++(3.2475952641916446pt,5.625pt);
\draw[line width=1pt,smooth,samples=100,domain=0.715:1.705] plot(\x,{0.00053*(\x)^(5)+0.098*(\x)^(4)-1.534*(\x)^(3)+7.318*(\x)^(2)-11.448*(\x)+6.58});
\draw[line width=1pt,smooth,samples=100,domain=3.15:4.3] plot(\x,{0.00053*(\x)^(5)+0.098*(\x)^(4)-1.534*(\x)^(3)+7.318*(\x)^(2)-11.448*(\x)+6.58});
\draw[line width=1pt,smooth,samples=100,domain=5.27:5.7] plot(\x,{0.00053*(\x)^(5)+0.098*(\x)^(4)-1.534*(\x)^(3)+7.318*(\x)^(2)-11.448*(\x)+6.58});
\draw [line width=.75pt,dotted] (0,0.93)-- (6,0.93);
\draw[fill=black] (1.17,0.93) circle (1pt);
\draw[color=black] (1.18,0.6) node {$\left(0,z_{\mu}^{\lambda_2}\right)$};
\draw[color=black] (3.18,0.6) node {$\rightarrow$  Ground State};
\draw[color=black] (-0.8,1) node {$\frac{1}{N}\mathcal{S}^{\frac{N}{2}}(\lambda_2)$};
\draw [line width=.75pt,dotted] (0,2.2)-- (6,2.2);
\draw [fill=black] (5.475,2.2) circle (1pt);
\draw[color=black] (6,2.5) node {$\left(z_{\mu}^{\lambda_1},0\right)$};
\draw[color=black] (-0.8,2.3) node {$\frac{1}{N}S^{\frac{N}{2}}(\lambda_1)$};
\draw [line width=.75pt,dotted] (0,4)-- (6,4);
\draw[color=black] (-1.35,4.2) node {$\frac{2}{N}\!\left(\frac{\mathcal{S}(\lambda_1)+\mathcal{S}(\lambda_2)}{2}\right)^{\frac{N}{2}}$};
\draw [line width=.75pt,dotted] (0,5.45)-- (6,5.45);
\draw[color=black] (-0.45,5.45) node {$c_{MP}$};
\draw [fill=black] (3.74,5.45) circle (1pt);
\draw[color=black] (3.8,5.75) node {$\left(\tilde{u},\tilde{v}\right)$};
\draw[color=black] (5.5,-0.3) node {$\|(\cdot \, ,\cdot)\|_{\mathbb{D}}$};
\draw[color=black] (-0.5,6.5) node {$\mathcal{J}_{\nu}{\big|}_{\mathcal{N}_{\nu}}$};
\end{tikzpicture}
\caption{The energy configuration given by Theorem~\ref{thm:groundstates} and Theorem~\ref{thm:MPgeom}.}
\end{figure}  
\end{proof}

\begin{center}{\bf Acknowledgements}\end{center} The authors wishes to thank B. Abdellaoui for useful discussions concerning the problem. \\ This work has been partially supported by the Madrid Government (Comunidad de Madrid-Spain) under the Multiannual Agreement with UC3M in the line of Excellence of University Professors (EPUC3M23), and in the context of the V PRICIT (Regional Programme of Research and Technological Innovation).\\ The authors are partially supported by the Ministry of Economy and Competitiveness of Spain, under research project PID2019-106122GB-I00.

\end{document}